\documentclass[11pt]{article} 
\usepackage{latexsym}
\usepackage[margin=30pt,font=small,labelfont=bf, labelsep=endash]{caption}
\usepackage{amsthm}
\usepackage{amsmath}
\usepackage{amssymb}
\usepackage{setspace}
\usepackage{array,booktabs}
\usepackage{eucal}
\usepackage{graphicx}
\usepackage[all, poly, knot]{xy}
\usepackage{verse, play}
\usepackage{graphicx, graphics}
\usepackage{amscd,amssymb,hyperref, latexsym}
\usepackage[all]{xy}
\usepackage{amssymb, amsmath}
\usepackage{epstopdf}
\input xy
\xyoption{all}
\xyoption{knot}
\xyoption{arc}

\newtheorem{theorem}{Theorem}
\newtheorem{lemma}[theorem]{Lemma}

\newtheorem{proposition}[theorem]{Proposition}

\title{The Region Smoothing Swap Game}
\author{Allison Henrich, Inga Johnson, Jonah Ostroff}
\begin{document}

\maketitle

\abstract{We introduce a topological combinatorial game called the Region Smoothing Swap Game. The game is played on a game board derived from the connected shadow of a link diagram on a (possibly non-orientable) surface by smoothing at crossings. Moves in the game are performed on regions of the diagram and can switch the direction of certain crossings' smoothings. The players' goals relate to the connectedness of the diagram produced by game play.}  

\section{Introduction.}
We introduce a new game, the \textbf{Region Smoothing Swap Game}, played on a game board derived from the connected shadow of a link diagram on a surface by smoothing at the diagram's crossings. This game can be viewed as a hybrid of the Link Smoothing Game, studied in \cite{link-smooth}, and the Region Unknotting Game, studied in \cite{RUG}. 

Suppose that $D$ is a connected diagram of a link shadow on some (possibly non-orientable) surface $S$. That is, $D$ is a connected link diagram where under- and over-strand information is unspecified at the crossings.  The unspecified crossings are called {\bf precrossings}.  There are two possible options for smoothing each precrossing, shown in  Figure~\ref{smooth}. When every precrossing of a link diagram $D$ has been smoothed, the result is called a {\bf smoothed state} or a {\bf smoothing} of $D$. A  {\bf region smoothing swap} is an operation that replaces one smoothed state by the state corresponding to changing the smoothings of all crossings on the boundary of a particular region in the link diagram. Figure~\ref{exampleGameBoard} shows a link together with two of its connected smoothed states. (The locations of the smoothed precrossings are indicated in the diagram by grey disks.) These two states are related by a single region smoothing swap, performed on the highlighted region.  
\begin{figure}[htbp] 
\begin{center} \includegraphics[scale=.3]{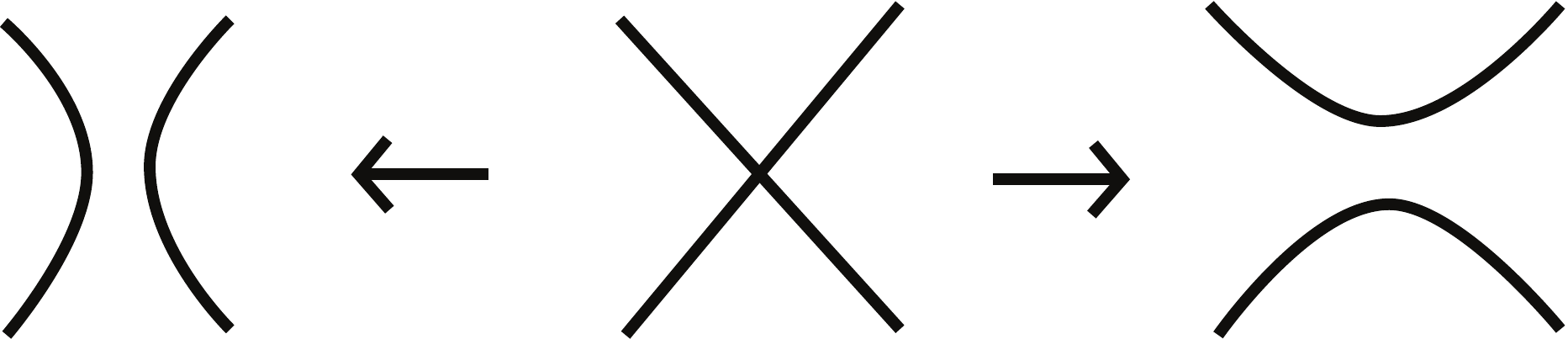}\end{center}
\caption{{\bf Smoothing a precrossing.}}
\label{smooth}
\end{figure}

\begin{figure}[htbp] 
\begin{center} \includegraphics[scale=.25]{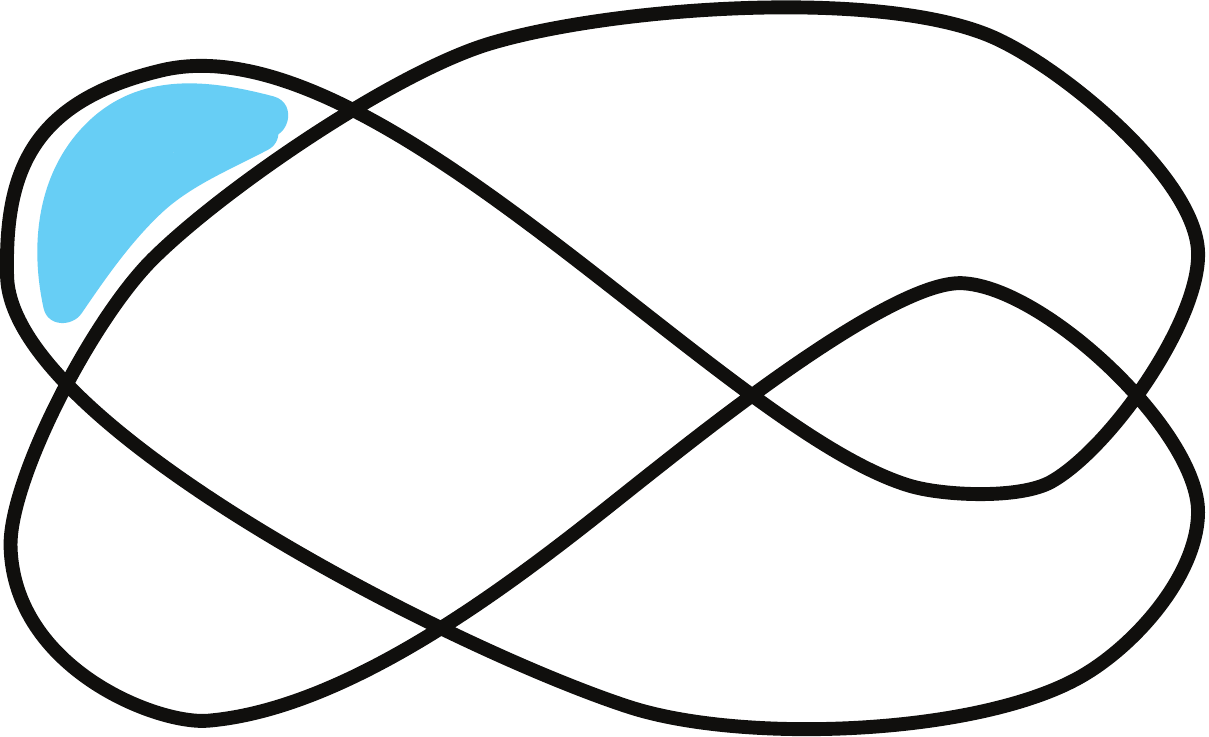} \\
\includegraphics[scale=.2]{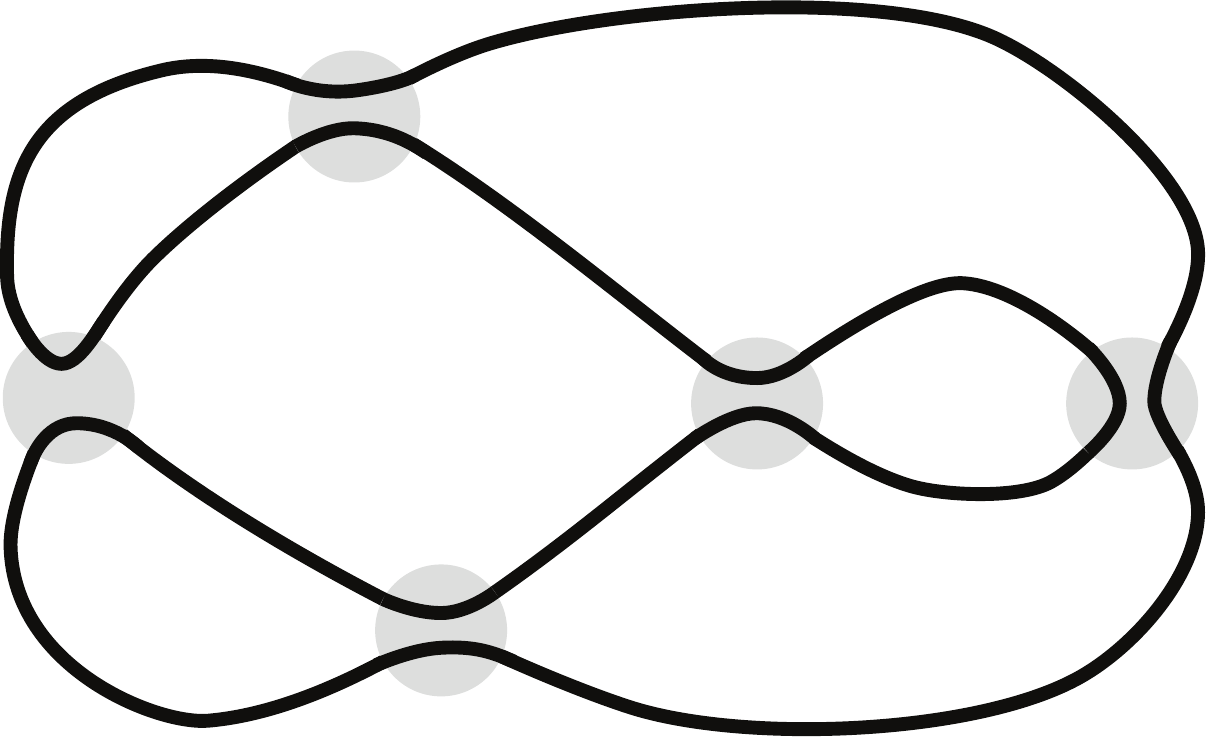} \hspace{.5in} \includegraphics[scale=.2]{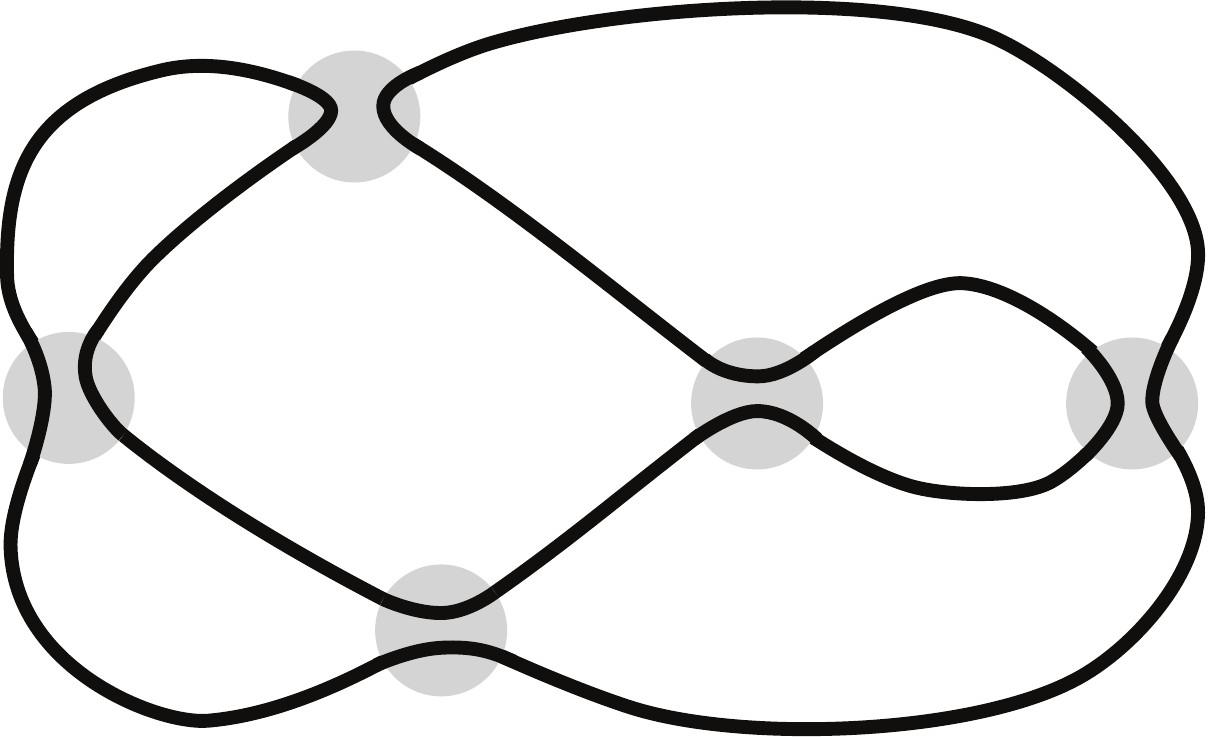}
\end{center}
\caption{{\bf The shadow of a twist knot and two connected smoothed states of the diagram. Unless otherwise specified, we assume that our diagrams lie on a sphere.}}
\label{exampleGameBoard}
\end{figure}

Game play begins on a connected smoothed state of link diagram $D$ on a surface $S$. Two players take turns making moves. During game play, the smoothed diagram may become disconnected. A move consists of choosing a region in the link diagram and either performing a region smoothing swap at that region or keeping the state as-is.  The two players continue selecting regions and moving until all regions of $D$ on $S$ have been selected exactly once.  One player plays with the goal of producing a connected smoothed state when game play ends while the other player wants the game to end with a disconnected diagram.  We will call the player that wants to keep the diagram {\bf c}onnected {\bf Player  C} and the player whose goal is to {\bf d}isconnect the diagram {\bf Player D}.  

An example game is played in Figure~\ref{gameplay}.  As each region is selected, the blue shading indicates that there is no change to the smoothings along the boundary of the region; the green shading indicates the adjacent smoothings are switched.

\begin{figure}[htbp] \begin{center}
\hspace{-0.4in}\begin{tabular}
{c *5{>{\centering\arraybackslash}m{0.6in}} @{}m{0pt}@{}}    
Begin & \includegraphics[scale=.18]{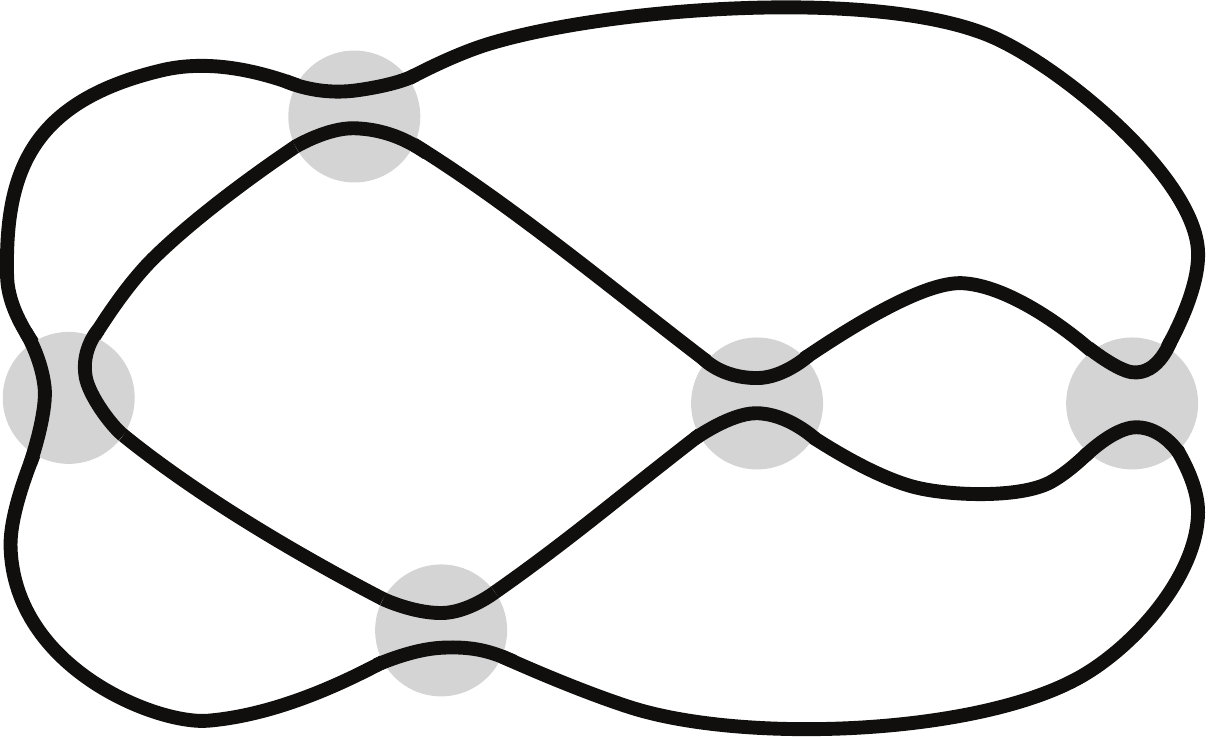}& \hspace{0.3in} &Move 4,  {\bf D} &\includegraphics[scale=.18]{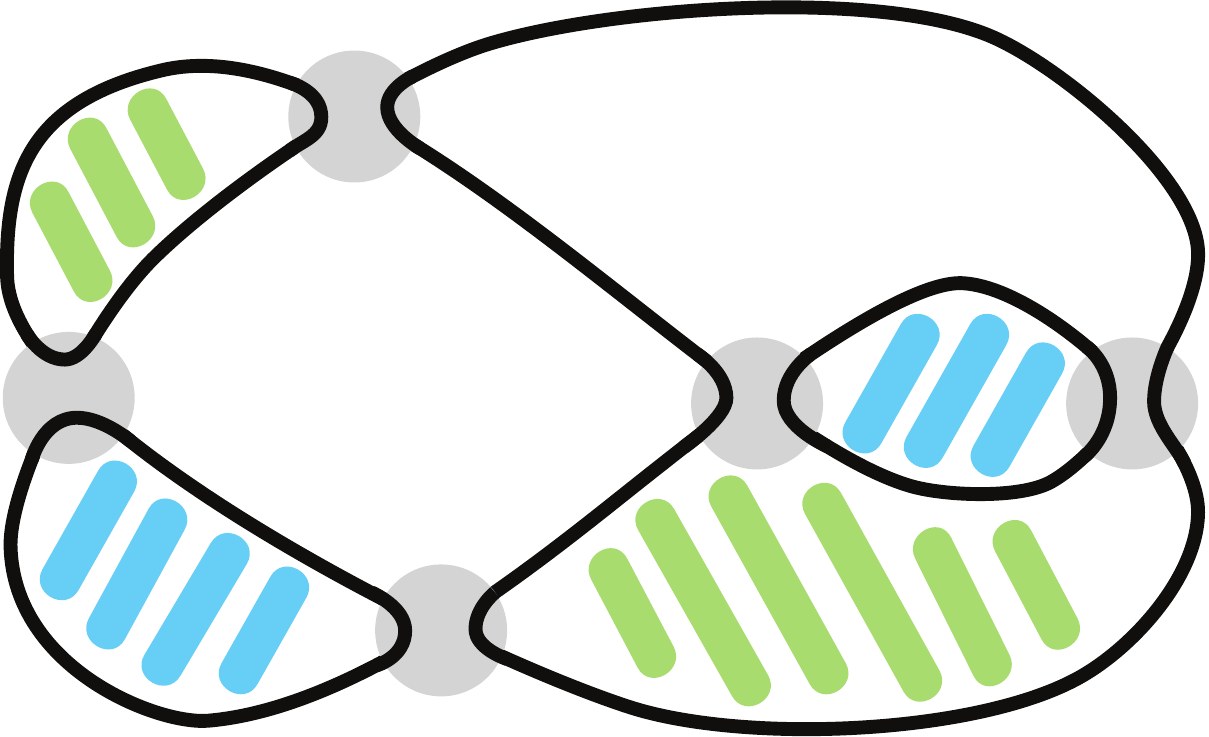}\\
Move 1, {\bf C}& \includegraphics[scale=.18]{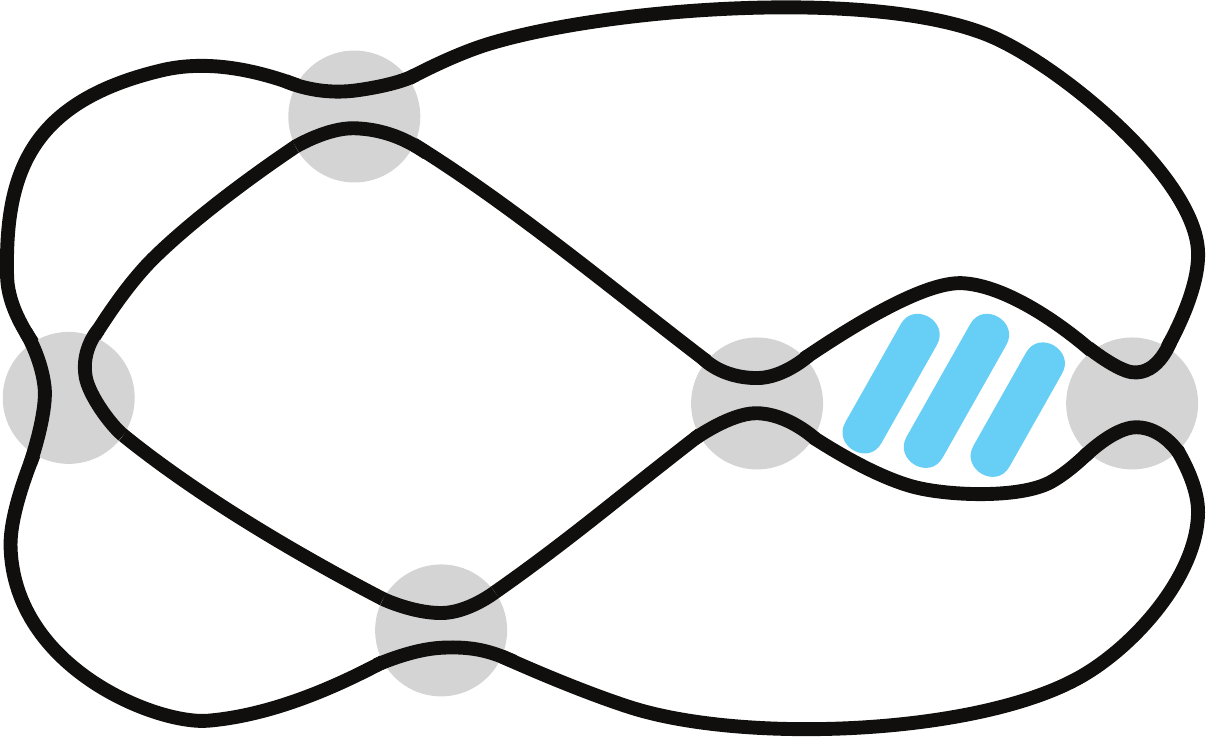}& &Move 5, {\bf C}&\includegraphics[scale=.18]{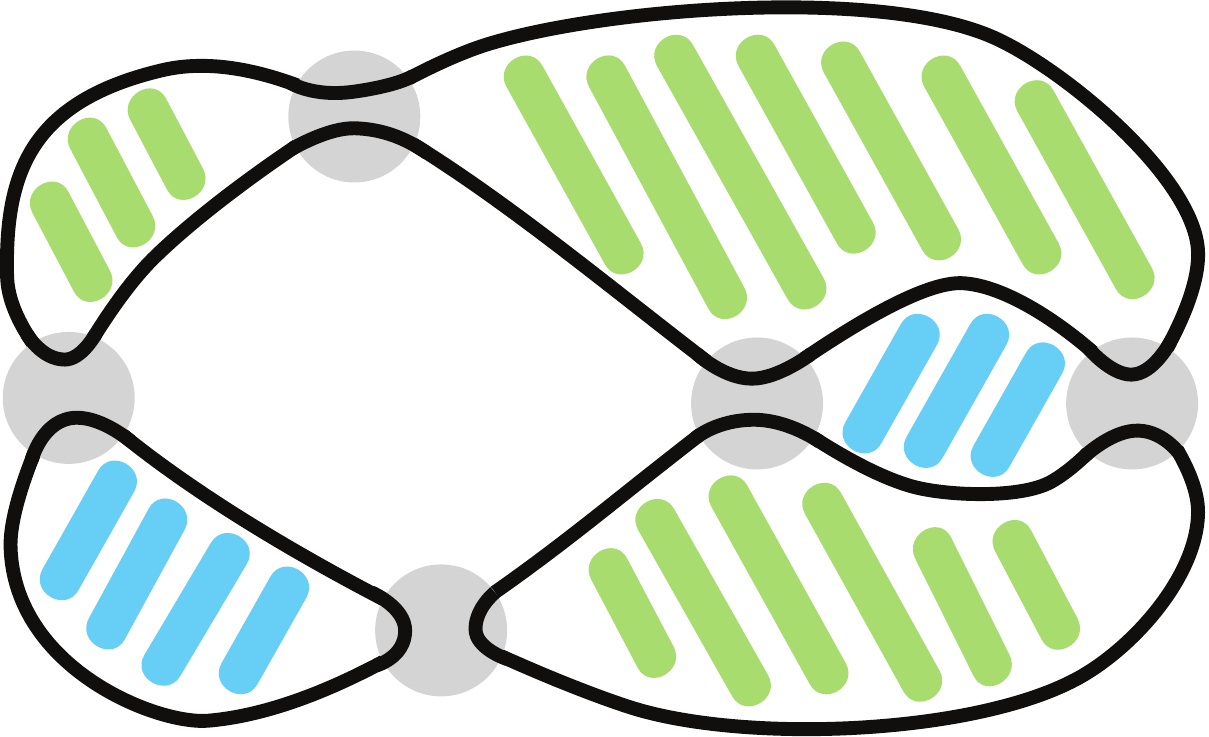}\\
Move 2, {\bf D}& \includegraphics[scale=.18]{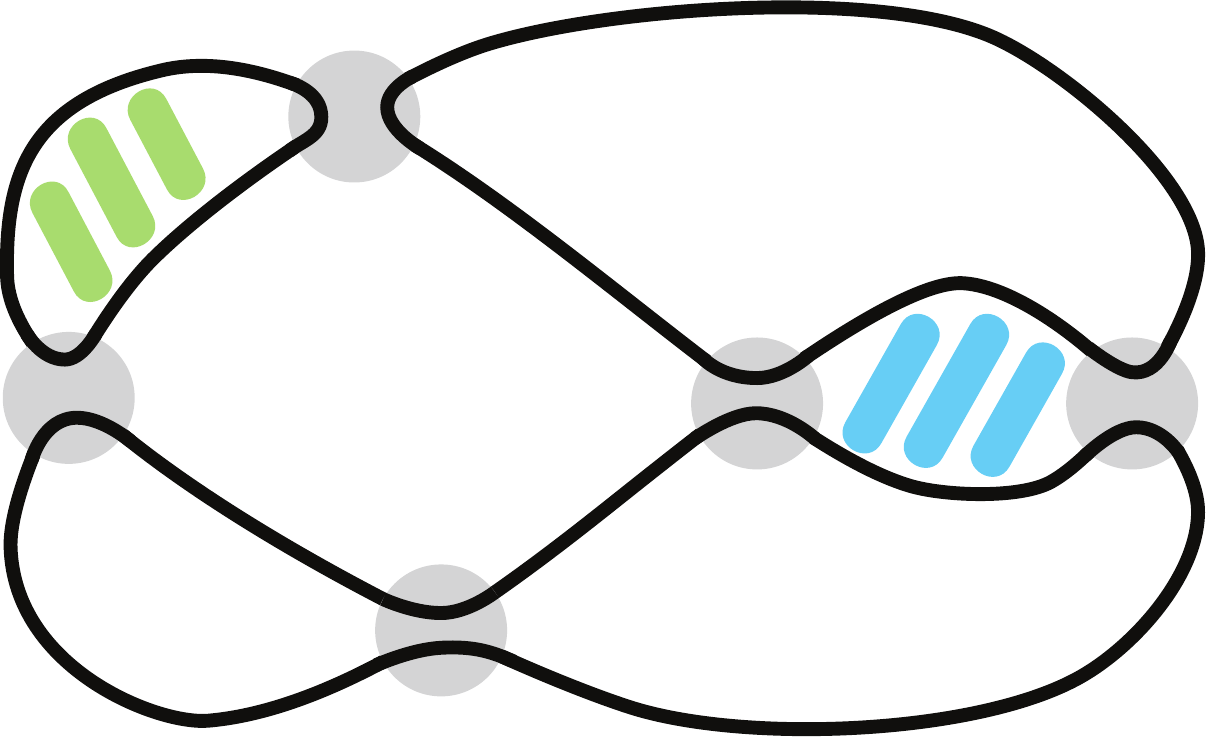}& &Move 6,  {\bf D}&\includegraphics[scale=.18]{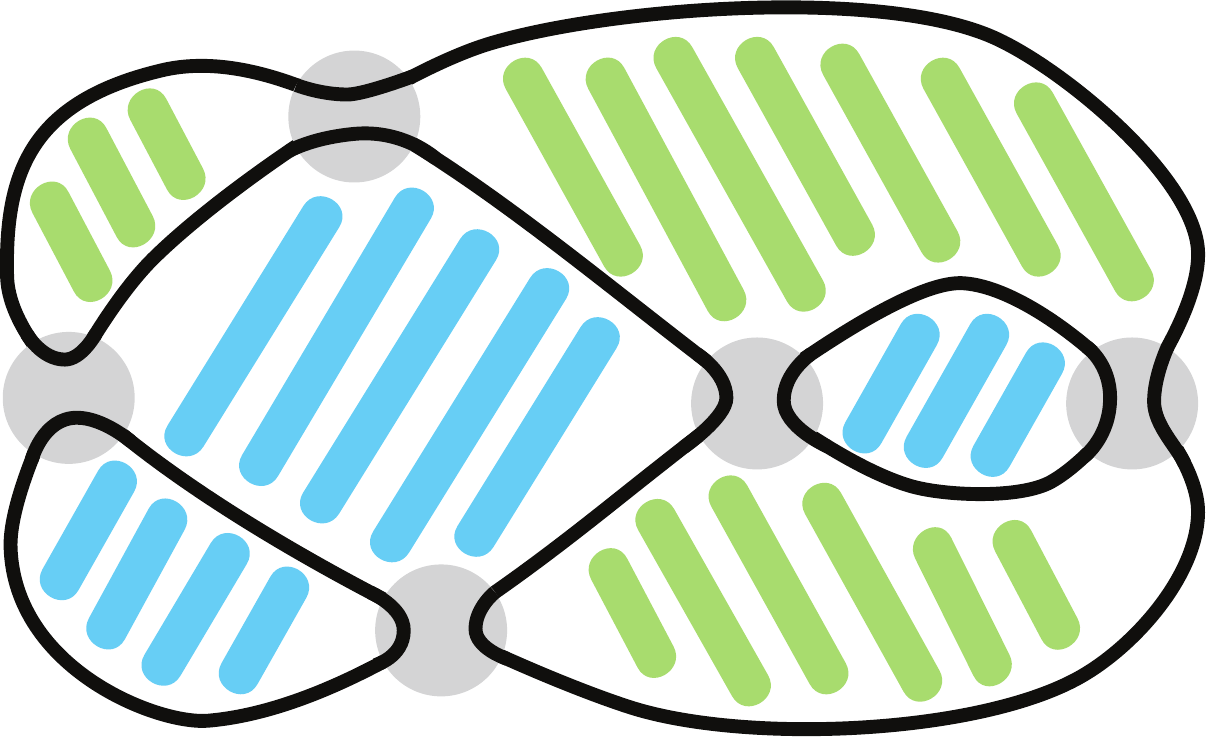}\\
Move 3,  {\bf C}& \includegraphics[scale=.18]{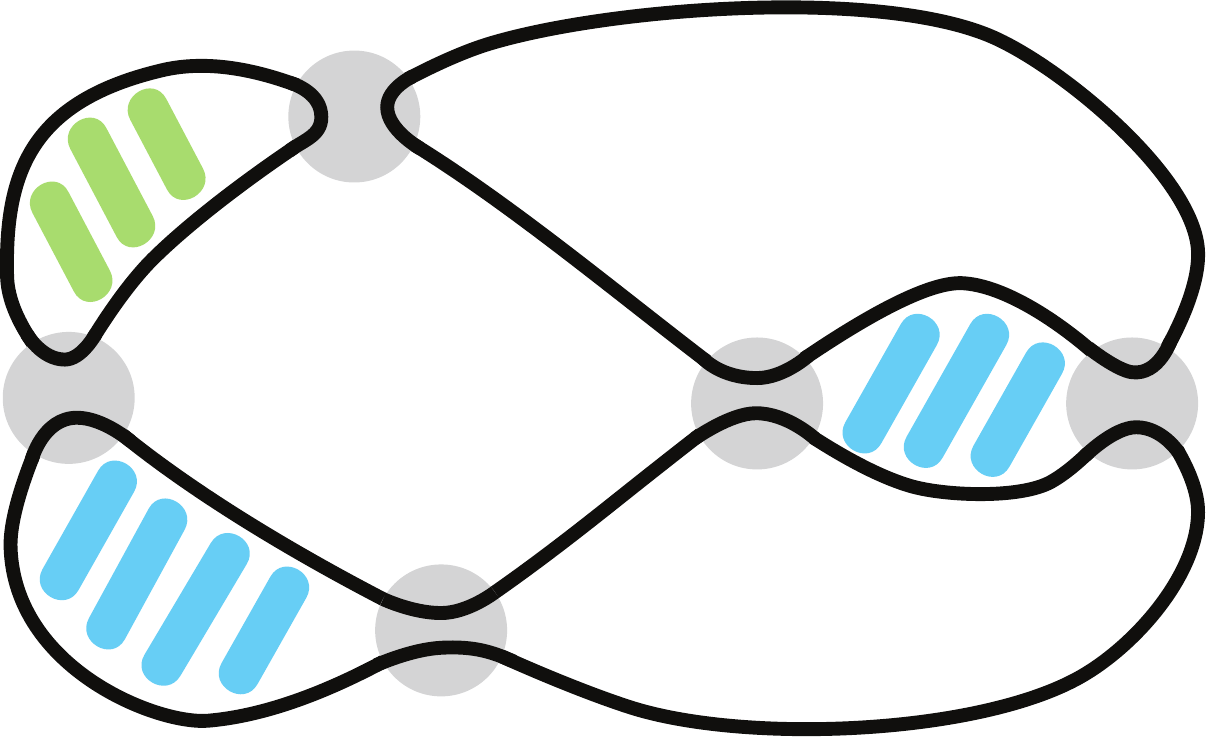}& &Move 7, {\bf C}&\includegraphics[scale=.18]{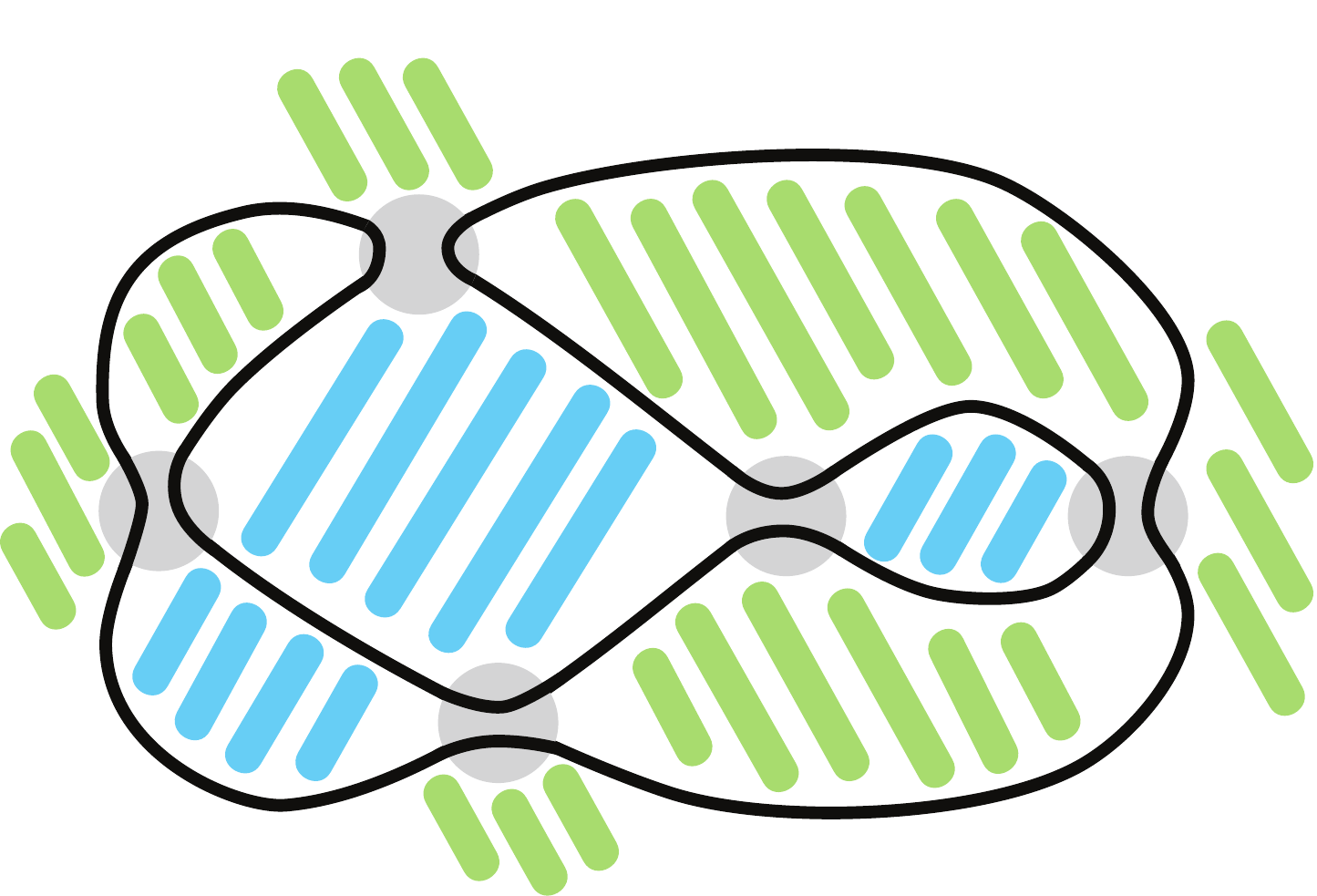}
\end{tabular}
\caption{{\bf The Region Smoothing Swap Game is played on a 5-crossing knot shadow that lies on a sphere. C plays first and wins.
}}
\label{gameplay}\end{center}
\end{figure}

In this paper, our goal is to classify connected smoothed states of link shadows on various surfaces according to which player has a winning strategy in the Region Smoothing Swap Game.  In general, Player D typically has a strong advantage in this game. This is because an arbitrary choice of smoothings for the crossings in a link diagram is far more likely to result in a state that has two or more connected components than in a state with a single component. Indeed, some links have equivalence classes of smoothed states (where equivalence is generated by the smoothing swap operation) that don't contain {\em any} connected states.   So, we aim to identify link shadows with smoothed states on which Player C has a winning strategy. We always begin with a game board that is connected to ensure that we are in an equivalence class that contains at least one connected state. 

The example in Figure~\ref{gameplay} shows game play on a link diagram and a particular smoothed state where C moves first and C wins.  In fact, on this particular game board, C can always win if she moves first. To understand why Player C has a winning strategy, we first translate the game into a game on graphs to simplify our analysis.

\section{Graphs and the Region Smoothing Swap Game.}\label{graphs}

It can be useful to shift from viewing our game on a link diagram to viewing it as a game on the corresponding checkerboard graph. In this new setting, the regions of the smoothed link diagram are represented by either a vertex {\it or} a face of the graph.  The smoothed crossings are represented by the edges of the graph.  We see that the connectedness of a subset of edges in the graph captures the connectedness of the corresponding link diagram. 

Let's recall how to form the checkerboard graph corresponding to a link shadow.  We start by checkerboard coloring the shadow, then placing a vertex in each shaded region. Whenever two shaded regions are connected by a precrossing, we connect the associated vertices with an edge. The smoothing choice in the initial connected smoothed link shadow is indicated by either turning an edge ``on" (pictured in bold) or ``off" (pictured as dashed) depending on whether the two adjacent regions are joined or separated by the chosen smoothing.  For example, in Figure~\ref{Board2Graph}, we show the initial game board from Figure~\ref{gameplay} and the associated graph with edges on or off according to the choice of smoothing in the diagram. 

\begin{figure}[htbp] \begin{center}
 \begin{tabular}{lcl}
 \includegraphics[scale=.3]{GamePlayMove0-eps-converted-to.pdf} & \hspace{0.4in} &  \includegraphics[scale=.3]{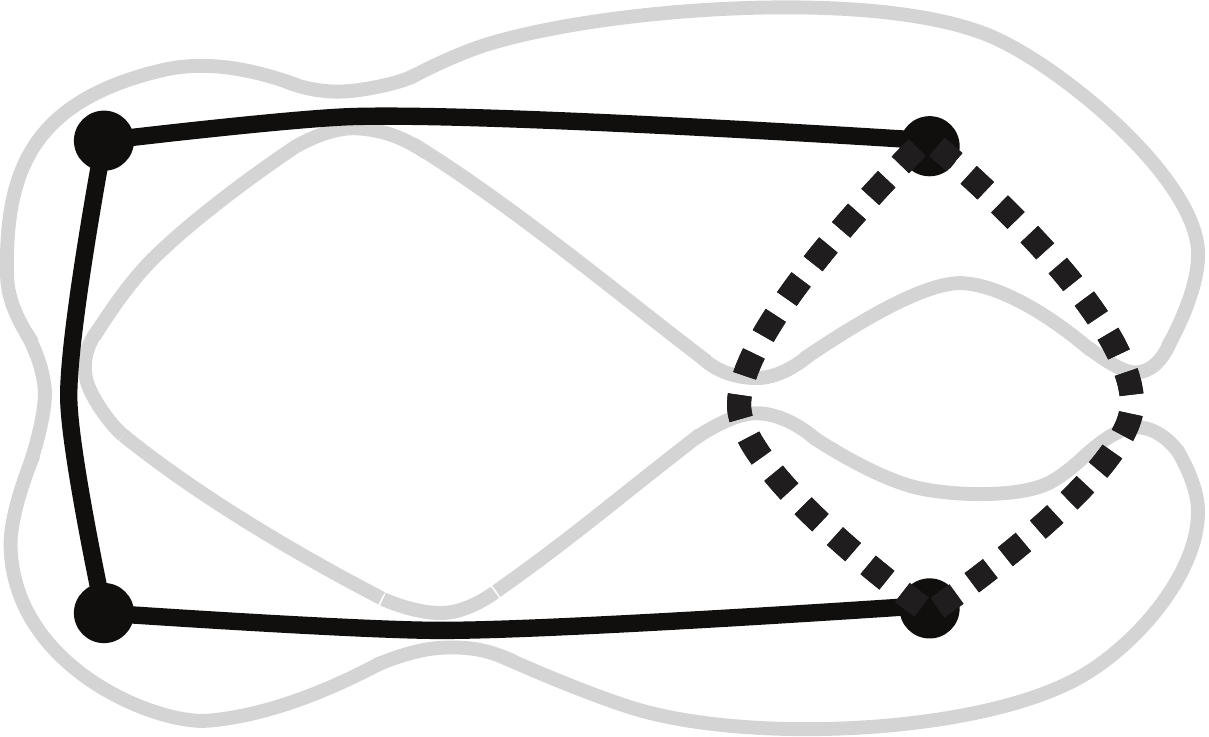}
\end{tabular}
\caption{{\bf Region Smoothing Swap Game board and the associated graph.}}
\label{Board2Graph}\end{center}
\end{figure}

In the graph setting, the moves are made both on the regions of the graph and on the vertices of the graph, since both represent regions in the smoothed diagram game board.  So, when a vertex of the graph is selected, a player may (1) switch every incident edge---i.e., turn those edges that are on to off and those that are off to on---or (2) keep all incident edges as-is.  When a region of the graph is selected, a player may (1) switch every edge along its boundary---from on to off and off to on---or (2) keep all edges along the boundary as-is. The connectedness of the associated smoothed state corresponds to the requirement that the subgraph of edges that are on (shown in bold) is a tree that contains every vertex.

One benefit of the graph model of the Region Smoothing Swap Game is that the moves and the current state of the game can be easily represented using vectors with integer entries modulo 2.  To encode the current state of the game, we begin by enumerating each precrossing of the link shadow $c_1, c_2, \dots, c_n$ and each region of the diagram $r_1, r_2, \dots, r_k$.  The associated graph will inherit the labeling of  $c_1, c_2, \dots, c_n$ on its edges and the labeling $r_1, r_2, \dots, r_k$ on its vertices and regions.  Observe that if the game board is assumed to lie on a sphere, then an Euler characteristic computation shows that $n + 2 =k$.

\

For a graph with $n$ edges, the initial game state vector ${V}_0=[v_1\text{ }v_2\text{ }\dots \text{ }v_n]$ is defined by $v_i=0$ if the $i^{th}$ edge is off, and  $v_i=1$ if the $i^{th}$ edge is on, for $1\leq i \leq n$.  We use the notation ${V}_j$ to denote the current game state vector after $j$ moves have been made, for $1\leq j \leq k$. 

A move on a region or vertex in the graph is denoted by $R_i = \epsilon_i r_i$  where $r_i= [m_{i,1}\text{ } m_{i,2}\text{ } \dots\text{ } m_{i,n}]$  is defined by $m_{i,j} =1$ if the edge $c_j$ is in the boundary of region $r_i$ or incident to vertex $r_i$, and  $m_{i,j} =0$ otherwise.  The value of $\epsilon_i $ is $1$ or $0$ depending on whether the player wants to switch all edge values or keep them as is.  The effect of the move $R_i$ on game state vector $V_j$  is $V_{j+1} = V_j+R_i$ modulo 2.  

\

 As an example, the game play from Figure~\ref{gameplay} is represented with vectors in Figure~\ref{VectorGameplay}.

\begin{figure}[htbp] \begin{center}
 \begin{tabular}{cm{1in} cm{2in} c l}
Begin  &\hspace{-0.45in} \includegraphics[scale=.23]{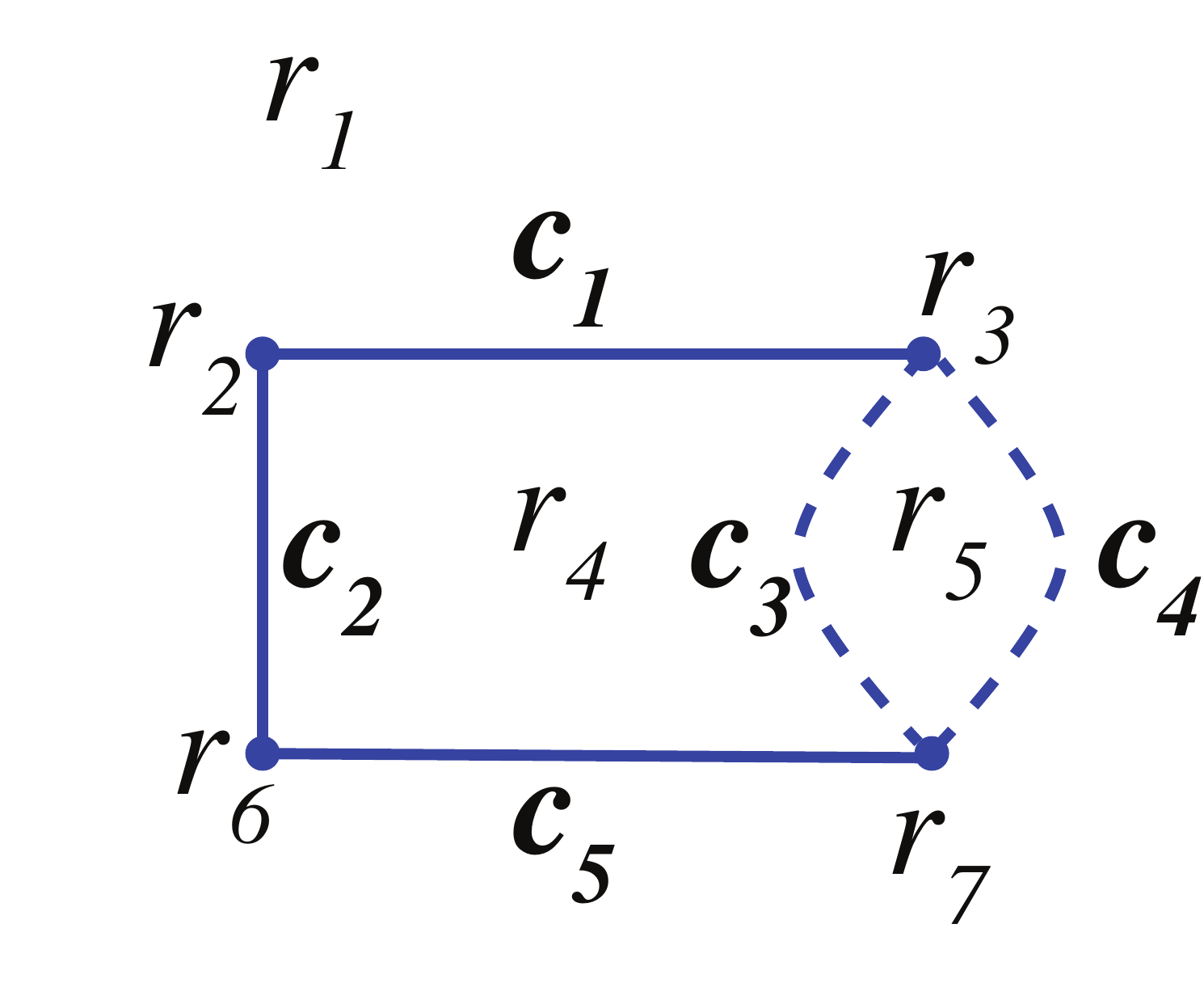}& & \begin{tabular}{crl} $\left.\right.$ \hspace{13pt}&\hspace{2pt} $V_0= $& \hspace{2pt} $[1\text{ } 1\text{ }  0\text{ } 0\text{ }  1] $\\ \end{tabular} \\
Move 1, {\bf C}& \includegraphics[scale=.23]{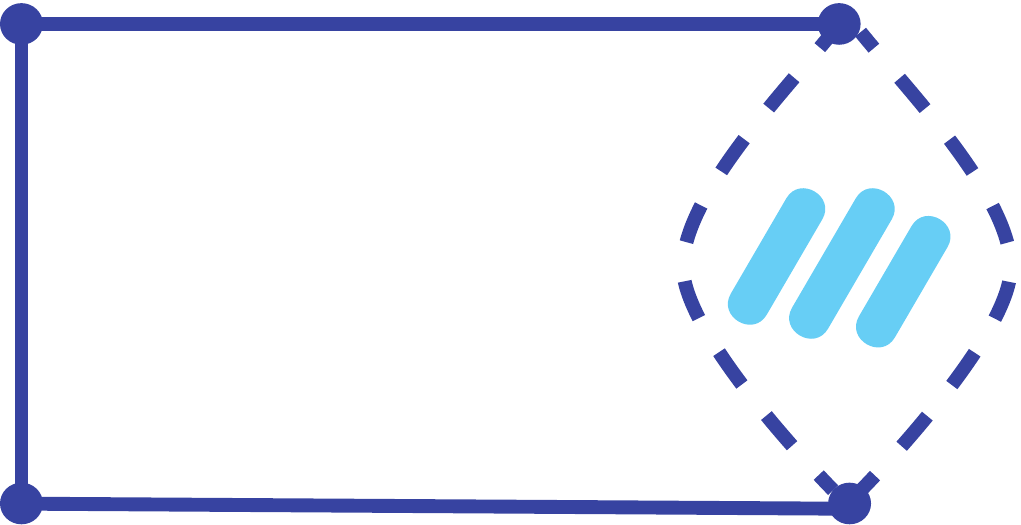}& &  \begin{tabular}{rrl} 
+& $R_5=$& $0 [0\text{ }  0\text{ }  1\text{ }  1\text{ }  0]$\\ \hline 
&$V_1 = $& \hspace{6pt}$[1\text{ }  1\text{ }  0\text{ }  0\text{ }  1] $\end{tabular}\\
Move 2, {\bf D}& \includegraphics[scale=.23]{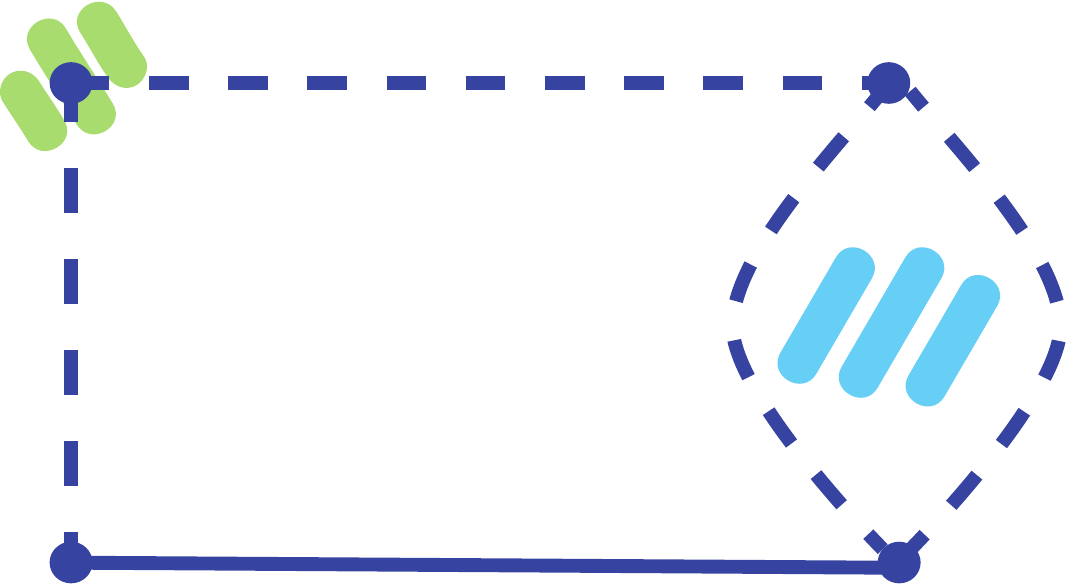}& & \begin{tabular}{rrl} 
+& $R_2=$& $1 [1\text{ }  1\text{ }  0\text{ }  0\text{ }  0]$\\ \hline \text{ }
&$V_2 = $& \hspace{6pt}$[0\text{ } 0\text{ } 0\text{ } 0\text{ } 1] $\end{tabular}\\
Move 3,  {\bf C}& \includegraphics[scale=.23]{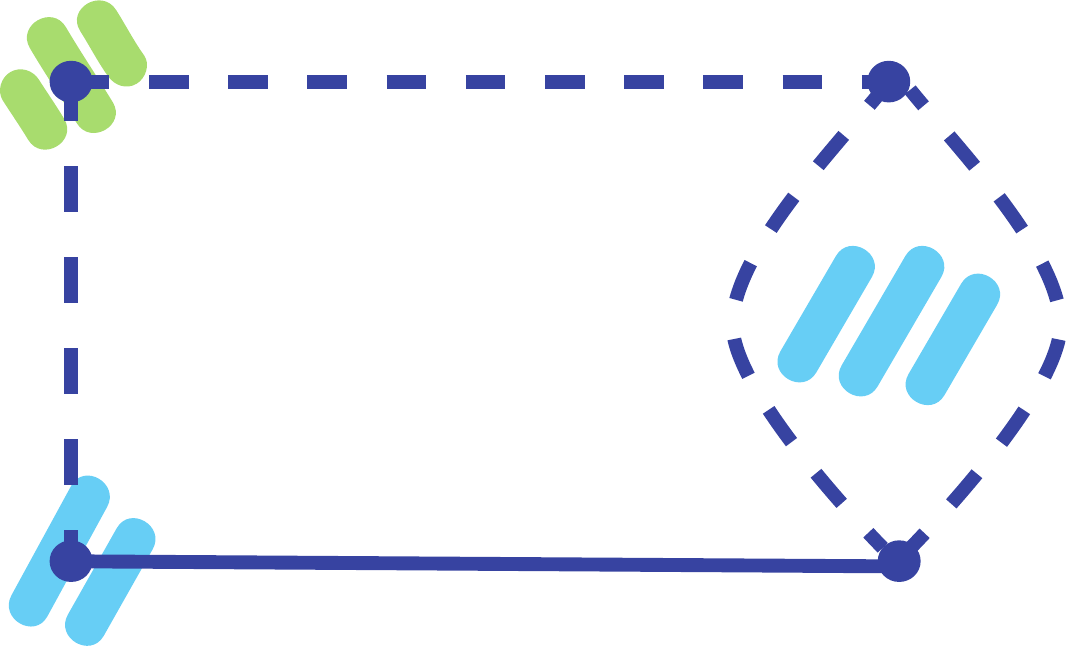} &&\begin{tabular}{rrl} 
+& $R_6=$& $0 [0\text{ } 1\text{ } 0\text{ } 0\text{ } 1]$\\ \hline 
&$V_3 = $& \hspace{6pt}$[0\text{ } 0\text{ } 0\text{ } 0\text{ } 1] $\end{tabular}\\
Move 4,  {\bf D} &\includegraphics[scale=.23]{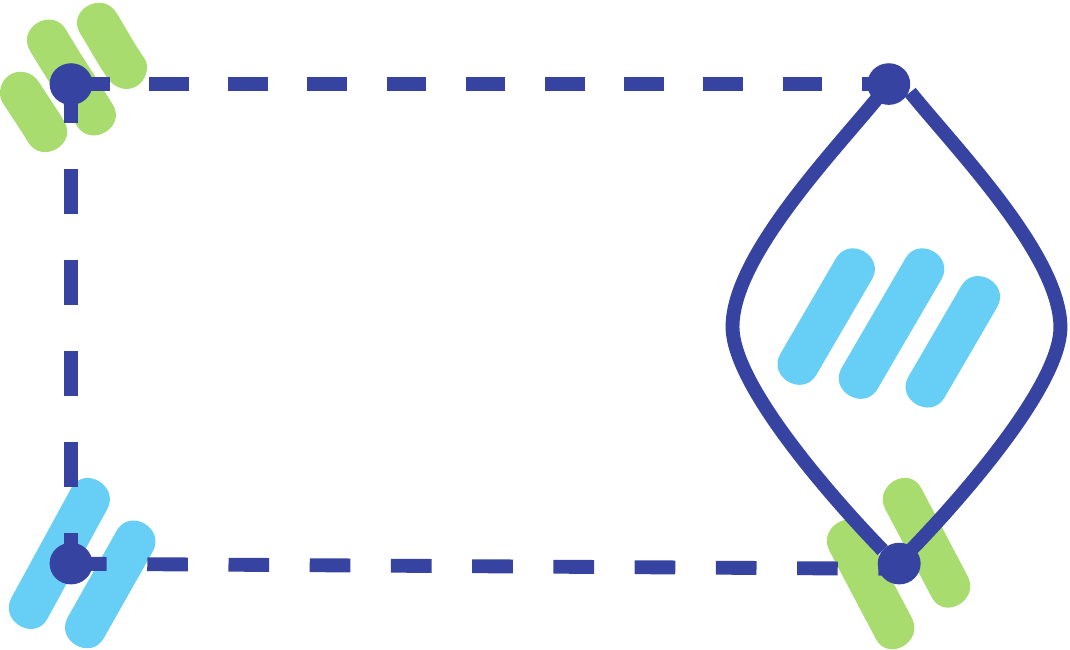}& &\begin{tabular}{rrl} 
+& $R_7=$& $1 [0\text{ } 0\text{ } 1\text{ } 1\text{ } 1]$\\ \hline 
&$V_4 = $& \hspace{6pt}$[0\text{ } 0\text{ } 1\text{ } 1\text{ } 0] $\end{tabular}\\
Move 5, {\bf C}&\includegraphics[scale=.23]{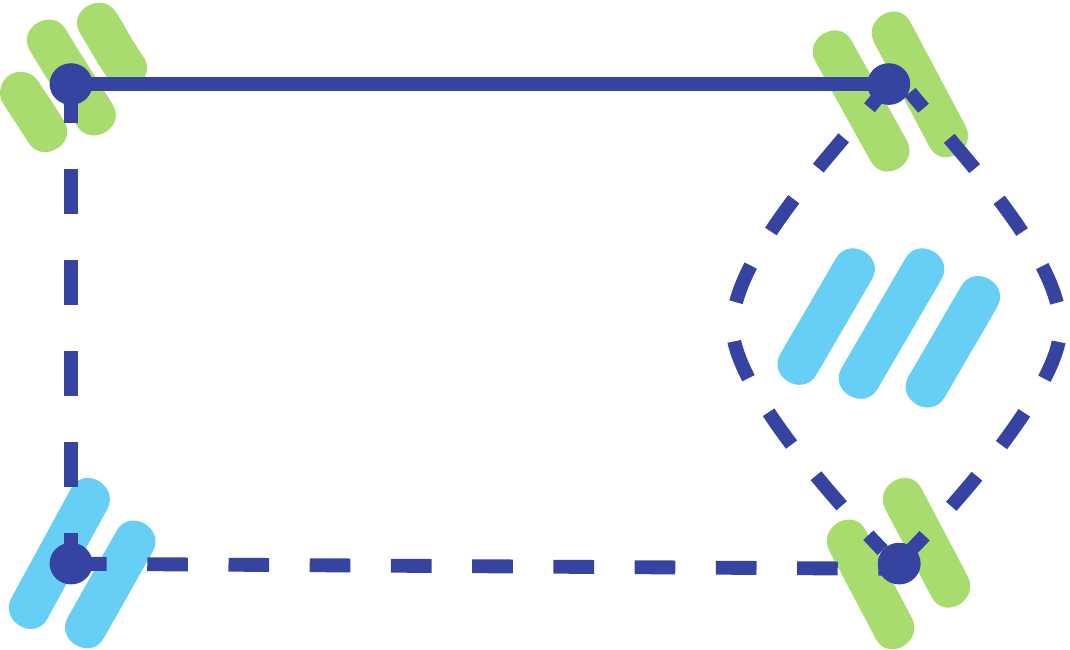}& &\begin{tabular}{rrl} 
+& $R_3=$& $1 [1\text{ } 0\text{ } 1\text{ } 1\text{ } 0]$\\ \hline 
&$V_5 = $& \hspace{6pt}$[1\text{ } 0\text{ } 0\text{ } 0\text{ } 0] $\end{tabular}\\
Move 6,  {\bf D}&\includegraphics[scale=.23]{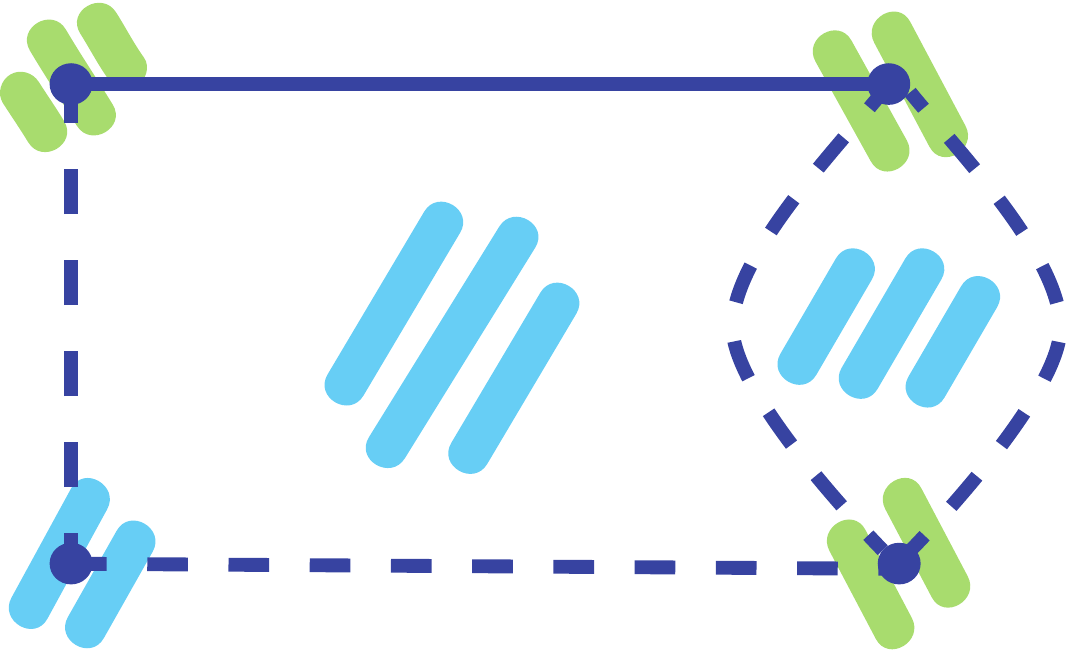}& &\begin{tabular}{rrl} 
+& $R_4=$& $0 [1\text{ } 1\text{ } 1\text{ } 0\text{ } 1]$\\ \hline 
&$V_6 = $& \hspace{6pt}$[1\text{ } 0\text{ } 0\text{ } 0\text{ } 0] $\end{tabular}\\
Move 7, {\bf C}&\includegraphics[scale=.23]{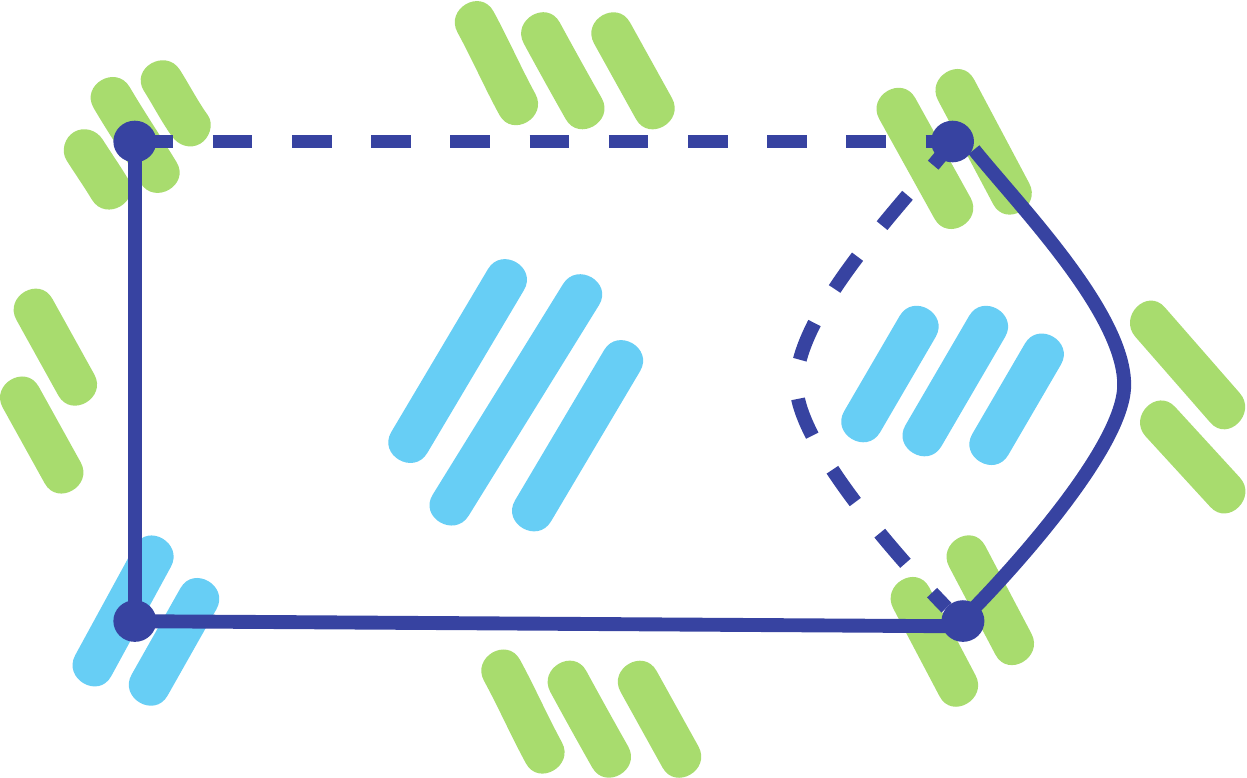}&& \begin{tabular}{rrl} 
+& $R_1=$& $1 [1\text{ } 1\text{ } 0\text{ } 1\text{ } 1]$\\ \hline 
&$V_7 = $& \hspace{6pt}$[0\text{ } 1\text{ } 0\text{ } 1\text{ } 1] $\end{tabular}
\end{tabular}
\caption{{\bf The graph version of the game shown in Figure~\ref{gameplay} with vector notation. C plays first and wins.
}}
\label{VectorGameplay}\end{center}
\end{figure}

By the commutativity of vector addition, the example in Figure~\ref{VectorGameplay} shows that the entire game play can be captured by the following matrix equation.

\begin{align*}
V_7&=V_0 + R_1 + R_2 + \dots + R_7 \\
&=V_0 + \epsilon_1 r_1 + \epsilon_2 r_2 + \dots + \epsilon_7 r_7\\
&=V_0 + [\epsilon_1 \hspace{0.1in}  \epsilon_2  \hspace{0.1in} \dots \hspace{0.1in} \epsilon_7  ]  \left[ \begin{tabular}{c}  $r_1 $\\  $r_2$  \\ \vdots \\  $r_7 $ \end{tabular} \right] 
\end{align*}

In general, given a connected state vector $V_0$, our goal is to strategically pick, or as we shall see {\it  pair}, the values $\epsilon_i$ for $1\leq i \leq k$ so that the resulting state vector $V_k$ from Equation \ref{eq} represents a connected diagram. 

\begin{equation}\label{eq} V_k=V_0 + [\epsilon_1 \hspace{0.1in}  \epsilon_2  \hspace{0.1in} \dots \hspace{0.1in} \epsilon_k  ]  \left[ \begin{tabular}{c}  $r_1 $\\  $r_2$  \\ \vdots \\  $r_k $ \end{tabular} \right]
\end{equation}

Using this formalism, we can determine a general winning strategy for Player C for the example in Figure~\ref{VectorGameplay}. Player C moves on $r_5$ first, keeping the region as-is. The remaining moves are determined by pairing moves on the following regions/vertices: $r_1$ and $r_4$,  $r_2$ and $r_6$, and $r_3$ and $r_7$.  In each pair, for any move made by D, there is a corresponding response move by C. In particular, Player C should respond to D's moves by choosing $\epsilon_1 \neq \epsilon _4$,  $\epsilon_2 \neq \epsilon _6$, and $\epsilon_3 = \epsilon _7$. Regardless of which vertices/regions Player D decides to move on and which signs for $\epsilon_i$ he chooses, C's response guarantees a connected diagram will result. Why is this? Consider the following matrix representing our game board.

\begin{gather}\label{r-matrix}
R=\begin{bmatrix}
r_1\\
r_2\\
r_3\\
r_4\\
r_5\\
r_6\\
r_7\\
\end{bmatrix}
=
\begin{bmatrix} 
1&1&0&1&1\\ 
 1&1&0&0&0\\
 1&0&1&1&0\\
 1&1&1&0&1\\
 0&0&1&1&0\\
 0&1&0&0&1\\
 0&0&1&1&1
\end{bmatrix}
\end{gather}

If Player C follows the strategy outlined above, the following are the eight possibilities for the vector $E=[\epsilon_1 \hspace{0.1in}  \epsilon_2  \hspace{0.1in} \dots \hspace{0.1in} \epsilon_7  ] $. The vectors represent the game play choices for Player D, with Player C's strategy entirely determined by D's choices. Note that the first vector in the list represents the game play shown in Figure \ref{VectorGameplay}.

\begin{gather*}
\begin{bmatrix} 1 & 1 & 1 & 0 & 0 & 0 & 1\end{bmatrix}\\
\begin{bmatrix} 1 & 1 & 0 & 0 & 0 & 0 & 0\end{bmatrix}\\
\begin{bmatrix} 1 & 0 & 1 & 0 & 0 & 1 & 1\end{bmatrix}\\
\begin{bmatrix} 1 & 0 & 0 & 0 & 0 & 1 & 0\end{bmatrix}\\
\begin{bmatrix} 0 & 1 & 1 & 1 & 0 & 0 & 1\end{bmatrix}\\
\begin{bmatrix} 0 & 1 & 0 & 1 & 0 & 0 & 0\end{bmatrix}\\
\begin{bmatrix} 0 & 0 & 1 & 1 & 0 & 1 & 1\end{bmatrix}\\
\begin{bmatrix} 0 & 0 & 0 & 1 & 0 & 1 & 0\end{bmatrix}\\
\end{gather*}

Multiplying each of these $E$ vectors by the $R$ matrix in Equation \ref{r-matrix} and adding the vector $V_0 =[1 \hspace{0.1in}  1  \hspace{0.1in} 0 \hspace{0.1in} 0 \hspace{0.1in} 1 ] $ gives the following eight vectors, the resulting values for $V_7$ that represent the states of the final game board.

\begin{gather*}
\begin{bmatrix} 0 & 1 & 0 & 1 & 1\end{bmatrix}\\
\begin{bmatrix} 1 & 1 & 0 & 1 & 0\end{bmatrix}\\
\begin{bmatrix} 1 & 1 & 0 & 1 & 0\end{bmatrix}\\
\begin{bmatrix} 0 & 1 & 0 & 1 & 1\end{bmatrix}\\
\begin{bmatrix} 0 & 1 & 1 & 0 & 1\end{bmatrix}\\
\begin{bmatrix} 1 & 1 & 1 & 0 & 0\end{bmatrix}\\
\begin{bmatrix} 1 & 1 & 1 & 0 & 0\end{bmatrix}\\
\begin{bmatrix} 0 & 1 & 1 & 0 & 1\end{bmatrix}
\end{gather*}

As we can see, there are only four distinct ending game states that are possible if D has freedom to choose his own moves and C follows her strategy. The corresponding graphs are shown in Figure \ref{endgame}. Note that the subgraphs defined by the edges that are turned on are all spanning trees in the checkerboard graph. Thus, the strategy we described is indeed a winning strategy.

\begin{figure}[htbp] \begin{center}
\includegraphics[scale=.25]{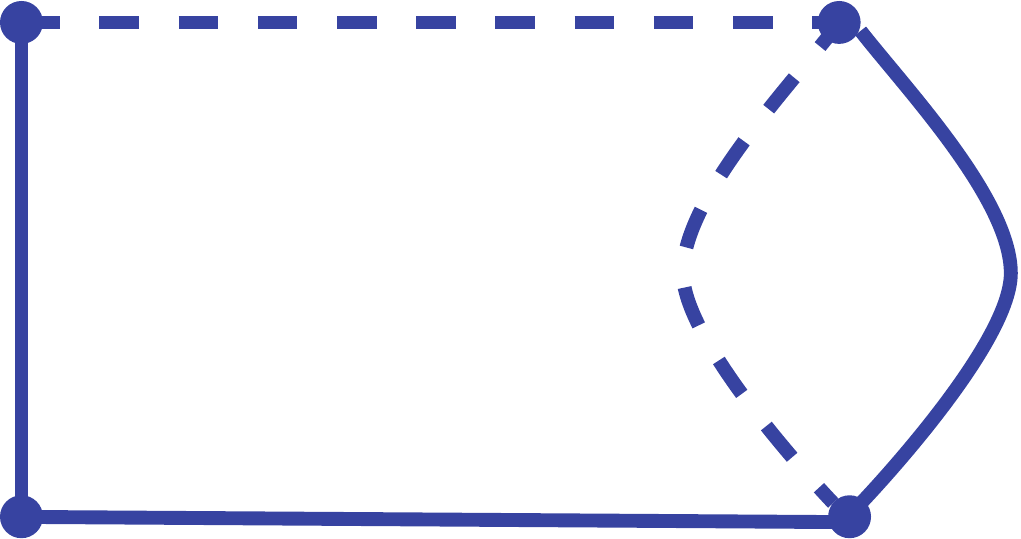}\hspace{1in}
\includegraphics[scale=.25]{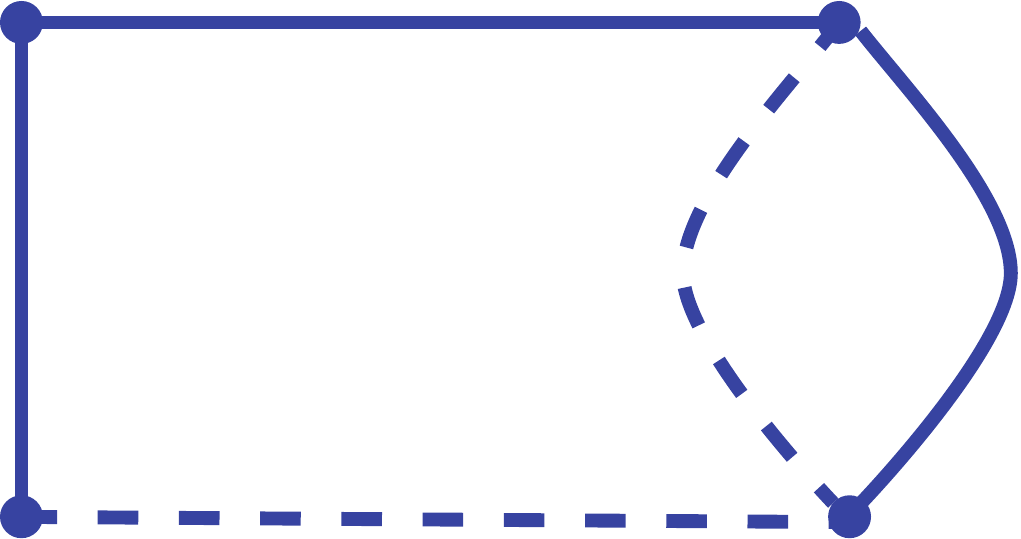}\\
$[0 \hspace{0.1in}  1  \hspace{0.1in} 0 \hspace{0.1in} 1  \hspace{0.1in} 1]$ \hspace{1.2in} $[1 \hspace{0.1in}  1  \hspace{0.1in} 0 \hspace{0.1in} 1  \hspace{0.1in} 0]$\\
\bigskip

\includegraphics[scale=.25]{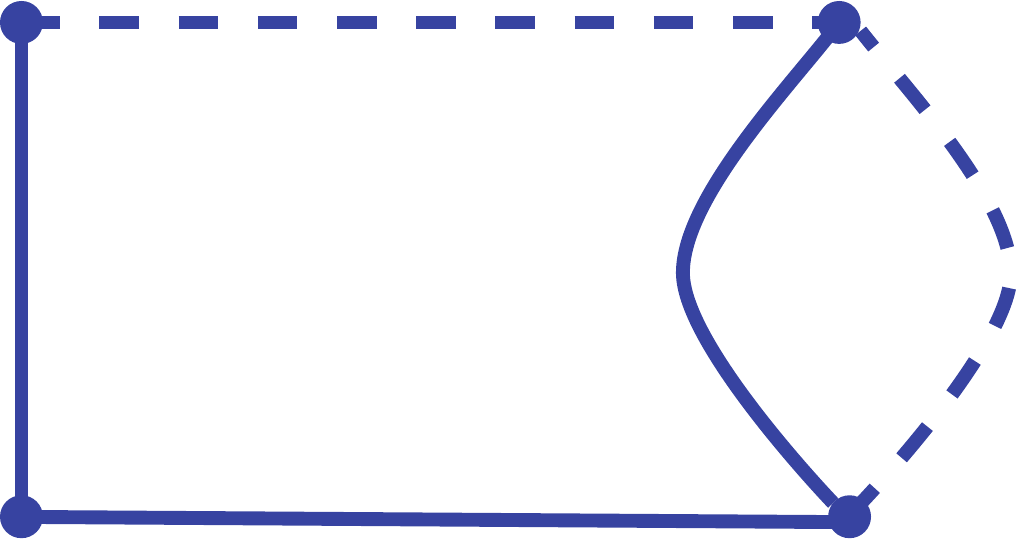}\hspace{1in}
\includegraphics[scale=.25]{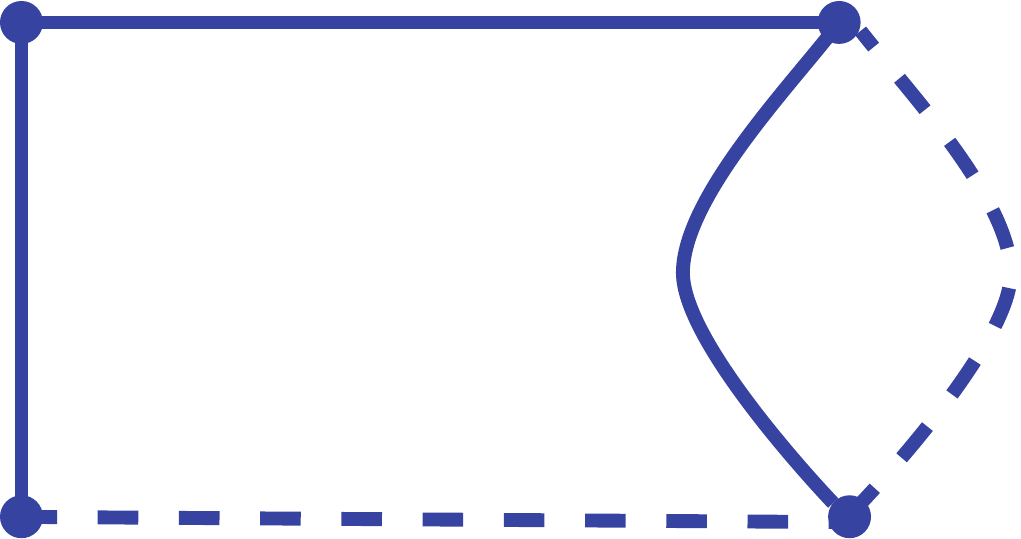}\\
$[0 \hspace{0.1in}  1  \hspace{0.1in} 1 \hspace{0.1in} 0  \hspace{0.1in} 1]$ \hspace{1.2in} $[1 \hspace{0.1in}  1  \hspace{0.1in} 1 \hspace{0.1in} 0  \hspace{0.1in} 0]$\\
\end{center}
\caption{{\bf The possible ending game boards when C follows her winning strategy on the graph in Figure \ref{Board2Graph}.
}}\label{endgame}
\end{figure}

\begin{figure}[htbp] \begin{center}
 \begin{tabular}{ccc}
(a)  \includegraphics[scale=.3]{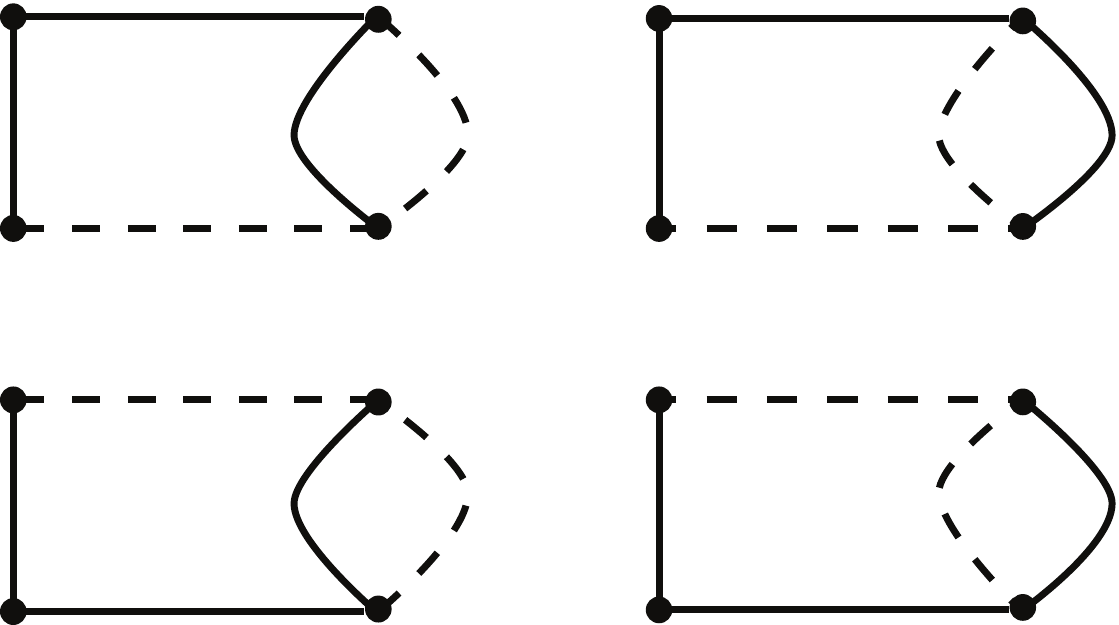} & (b) \includegraphics[scale=.3]{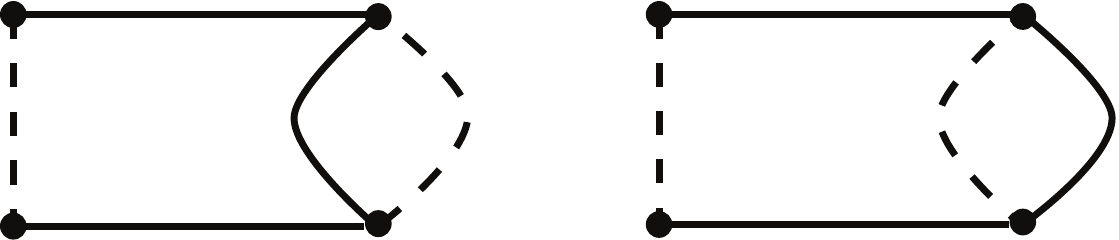}&  (c) \includegraphics[scale=.3]{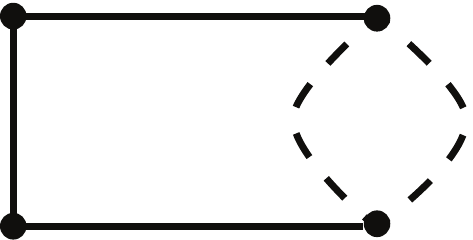}
\end{tabular}
\caption{{\bf All connected starting game boards on a 5-crossing twist knot shadow, grouped according to winning strategy. 
}}
\label{ConnectedBoards}\end{center}
\end{figure}

As we can see in Figure~\ref{exampleGameBoard}, there are many possible starting game boards that come from the twist knot.  In Figure~\ref{ConnectedBoards} we list {\it all} possible starting game graphs and group them according to their winning pairing strategy for C.  We will see that the pairs of moves described above remain quite useful on any starting game board.  In fact, the move pairing strategy can be used to prove the following theorem.

\begin{proposition} If Player C moves first, then Player C has a winning strategy playing the Region Smoothing Swap Game on any connected game board associated to the 5-crossing twist knot shadow in Figure~\ref{exampleGameBoard}.
\end{proposition}

\begin{proof}
All connected starting game boards are shown in Figure~\ref{ConnectedBoards} and placed in one of three groups labeled (a), (b), and (c).  For each  group of starting game boards, we describe a winning strategy for C, assuming she moves first. We use the names of regions, $r_i$, and crossings, $c_j$, as given in Figure~\ref{VectorGameplay}. In all groups, Player C's first move is on $r_5$ and she keeps the region as-is.  

For any starting board in group (a), we observe the following pairing strategy results in a win for C: $\epsilon_1 = \epsilon _4$,  $\epsilon_2 = \epsilon _6$, and $\epsilon_3 = \epsilon _7$.  In fact, after any single pair of moves $\epsilon_1 = \epsilon _4$,  $\epsilon_2 = \epsilon _6$, or $\epsilon_3 = \epsilon _7$, the mimicking strategy results in another game board within the list of Figure~\ref{ConnectedBoards}(a).  Thus, player C can defend the boards in group (a) so that after each move by D and corresponding reply by C, the board returns to a connected game board from group (a). 

For any starting board in group (b), we observe the following pairing strategy results in a win for C: $\epsilon_1 = \epsilon _4$,  $\epsilon_2 \neq \epsilon _6$, and $\epsilon_3 = \epsilon _7$.  Recall that, for any given set of moves, the order in which they are performed doesn't affect the outcome of the game. So, without loss of generality, we suppose the pair of moves  $\epsilon_2 \neq \epsilon _6$ are made first.  After these two moves are made on a game board from (b), the result is a game board from (a).  Then the remaining pairs of moves  $\epsilon_1 = \epsilon _4$ and $\epsilon_3 = \epsilon _7$ are applied to a game board from group (a), resulting in final board in group (a).

Lastly, as we saw in the more detailed linear algebraic argument above for the game board (c), the following pairing strategy results in a win for C: $\epsilon_1 \neq \epsilon _4$,  $\epsilon_2 \neq \epsilon _6$, and $\epsilon_3 = \epsilon _7$.  This pairing works regardless of the order in which the paired moves are applied.  When playing on the board in group (c), we note that if both pairs of moves $\epsilon_1 \neq \epsilon _4$ and $\epsilon_2 \neq \epsilon _6$ are completed first, then the result is a connected game board from (a).  The final pair of moves, $\epsilon_3 = \epsilon _7$, completes the game, producing a game board from (a).  
\end{proof}

We observe the special role played by the game boards in group (a) of Figure~\ref{ConnectedBoards}.  Once the starting board graph was moved into a board of type (a) by Player C, then C had a winning defensive strategy.  

\section{Game Play on the Sphere}

Three nice examples of link shadows on the surface of a sphere on which Player C has a winning strategy come from the minimum crossing diagrams of the figure-eight knot, the trefoil, and the Borromean rings.  The figure-eight knot in particular gives a nice example of a game that Player C can win using a mimicking strategy.

\begin{figure}[htbp] \begin{center}
 \begin{tabular}{ccc}
 \includegraphics[height=1.5in]{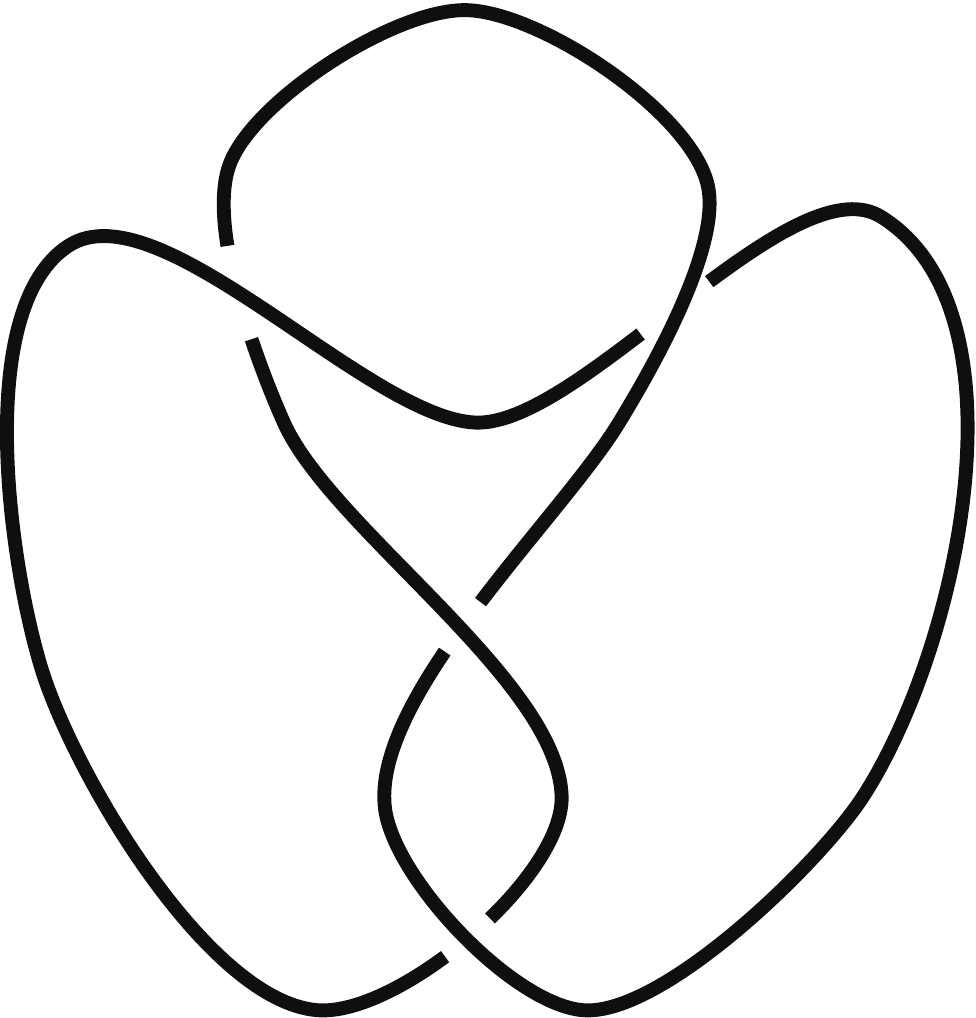} & & \includegraphics[height=1.5in]{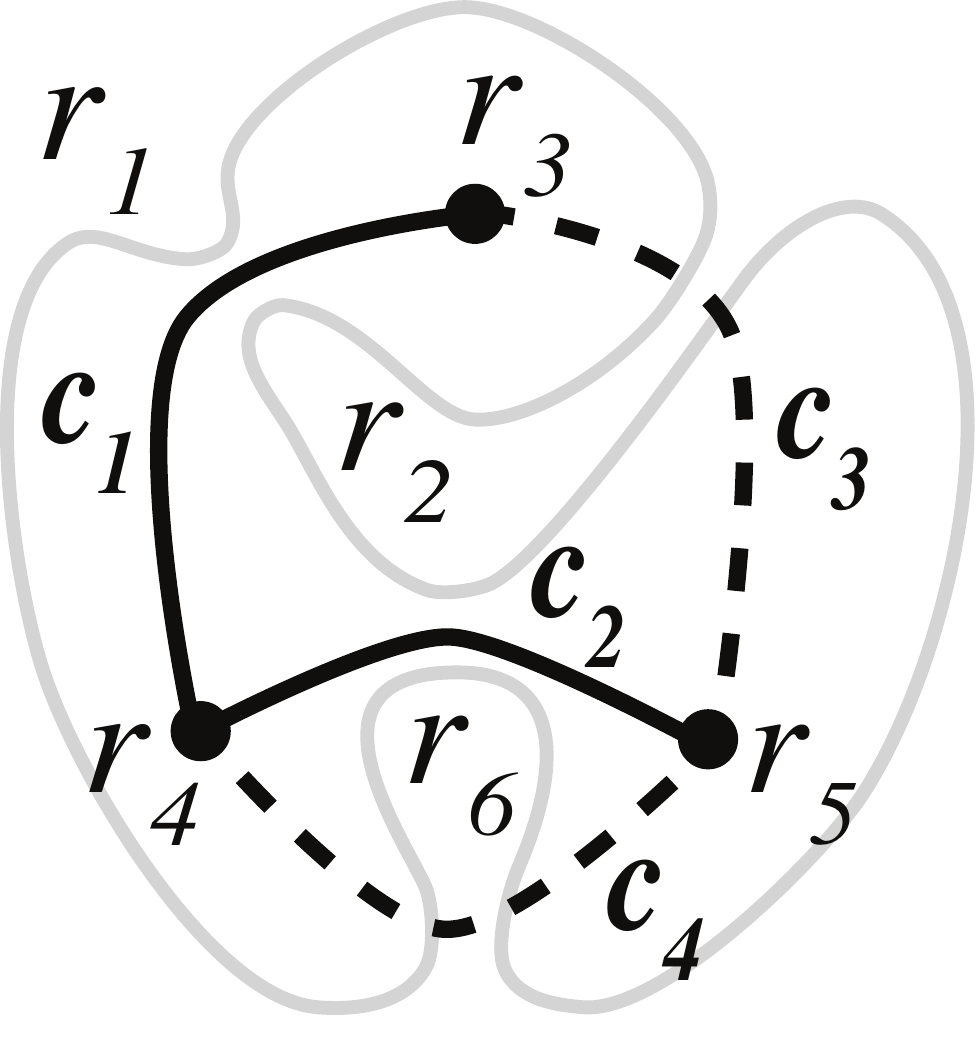}
\end{tabular}
\caption{\textbf{The figure-eight knot and a corresponding game board.}}
\label{fig8Game}\end{center}
\end{figure}

\begin{proposition} If Player C moves second, then Player C has a winning strategy playing the Region Smoothing Swap Game on any connected game board associated to the figure-eight knot shadow, as seen in Figure~\ref{fig8Game}.
\end{proposition}

\begin{proof}
For this proof, we refer to the labeling of regions/vertices by $r_i$ and crossings by $c_i$ shown in Figure~\ref{fig8Game}. The strategy Player C should follow to win is a mimicking strategy that pairs $r_1$ with $r_2$, $r_3$ with $r_6$, and $r_4$ with $r_5$. In other words, Player C should play so that $\epsilon_1=\epsilon_2$, $\epsilon_3=\epsilon_6$, and $\epsilon_4=\epsilon_5$. Following this strategy, given any game board from Figure~\ref{fig8Game-states}, any pair of moves yields another game board from this collection. For instance, the second graph is obtainable from the first in the figure (and vice versa) by changing both $r_1$ and $r_2$; the third graph is obtainable from the first by changing $r_4$ and $r_5$; the fourth graph is obtainable from the first by changing both $r_3$ and $r_6$. 

Since each member of the collection is a winning state for Player C, she can always win.
\end{proof}

\begin{figure}[htbp] \begin{center}
 \begin{tabular}{cccc}
 \includegraphics[height=.7in]{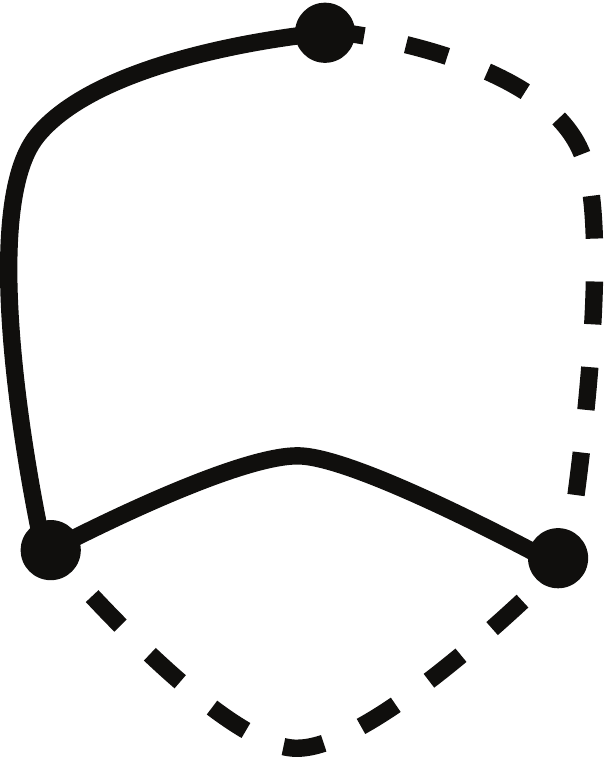} &  \includegraphics[height=.7in]{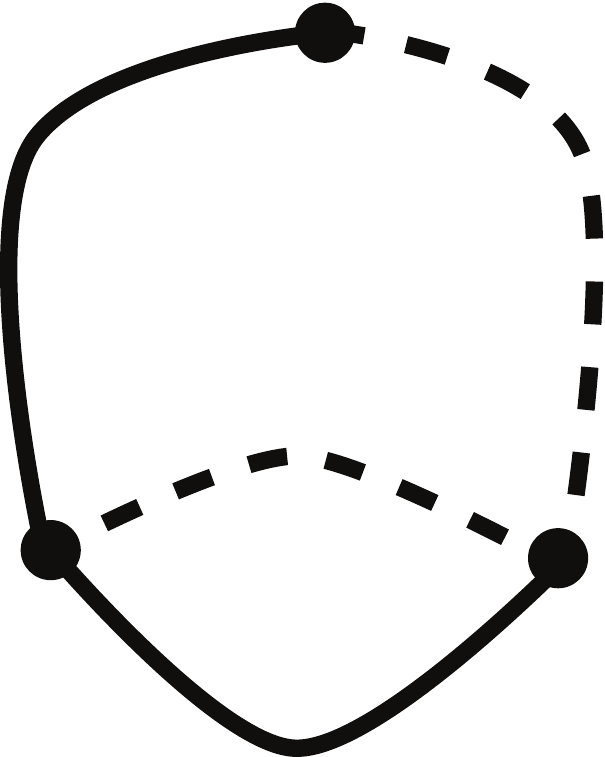}& \includegraphics[height=.7in]{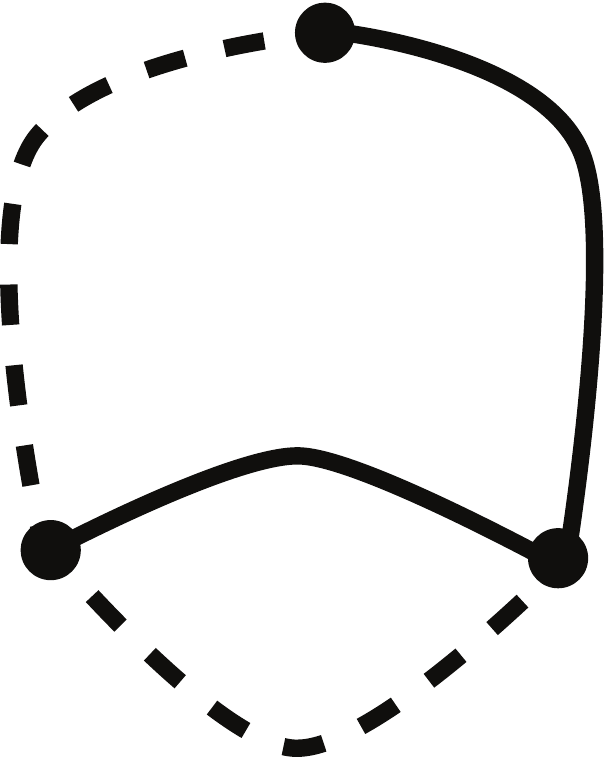} &  \includegraphics[height=.7in]{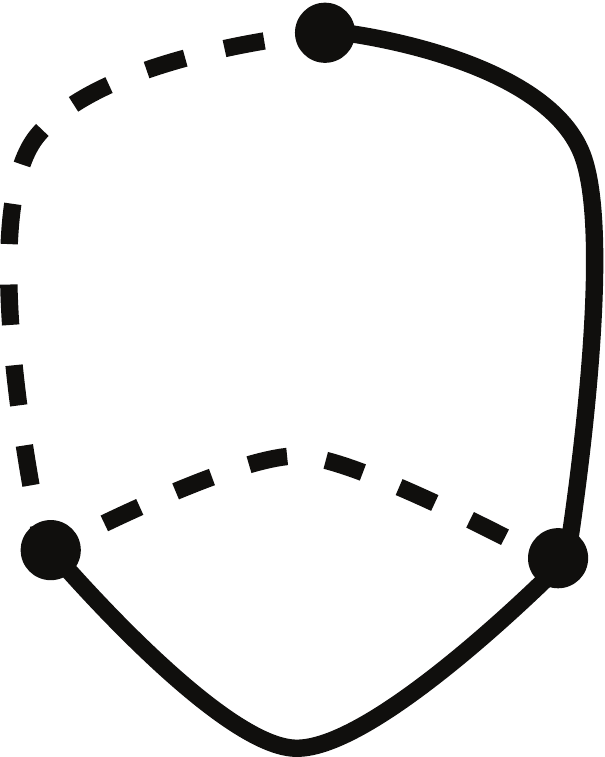}
\end{tabular}
\caption{{\bf Game states after Player C moves on the figure 8 knot.}}
\label{fig8Game-states}\end{center}
\end{figure}

\begin{figure}[htbp] \begin{center}
 \begin{tabular}{ccc}
\includegraphics[scale=.3]{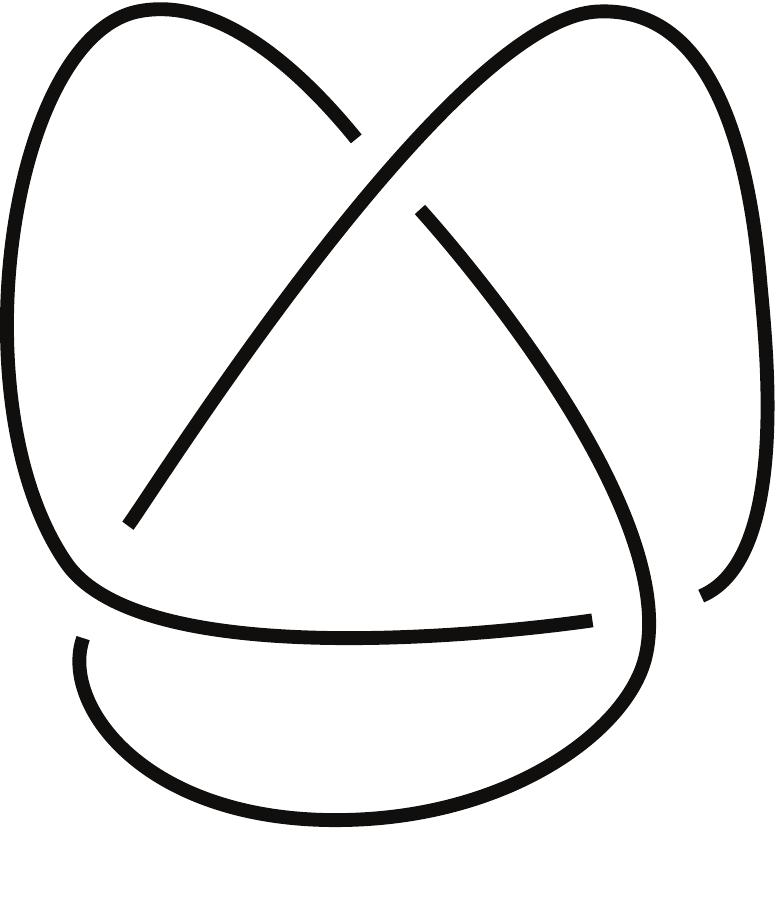}& &  \includegraphics[scale=.3]{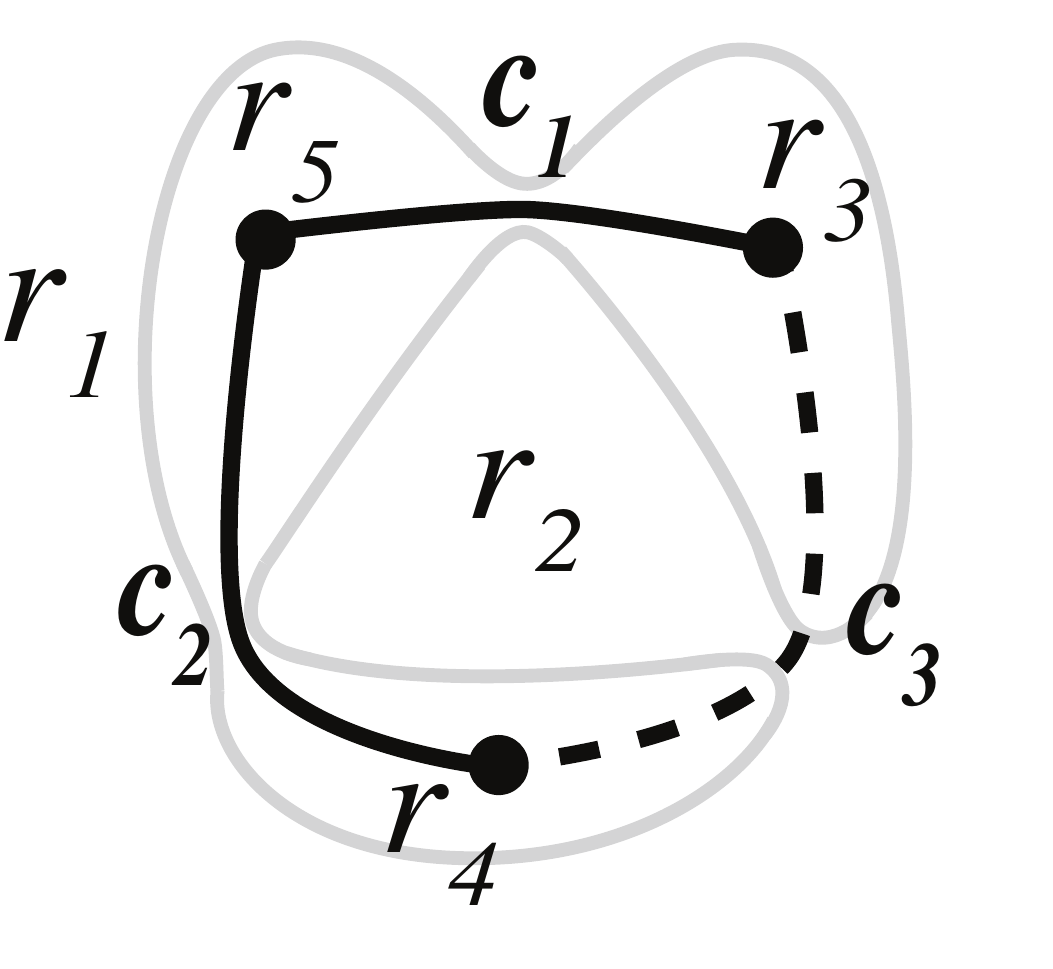} 
\end{tabular}
\caption{{\bf The trefoil and a corresponding game board.}}
\label{trefGame}\end{center}
\end{figure}

\begin{proposition} If Player C moves first, then Player C has a winning strategy playing the Region Smoothing Swap Game on any connected game board associated to the trefoil knot shadow in Figure~\ref{trefGame}.
\end{proposition}

\begin{proof}
For this proof, we refer to the labeling of regions/vertices using $r_i$ and crossings with $c_i$ as in Figure~\ref{trefGame}. Note that each connected game board is isomorphic to the one pictured, so it suffices to show that C has a winning strategy on this particular board. Let Player C move first on $r_5$, the region that contains the vertex of degree 2, keeping it as-is. For the remaining moves, we pair $r_1$ with $r_2$ and $r_3$ with $r_4$, and ensure that C follows a strategy such that $\epsilon_1=\epsilon_2$ while $\epsilon_3\neq\epsilon_4$. So, for instance, if D moves on $r_1$ and performs a smoothing swap, then C performs a smoothing swap on $r_2$. Any moves on $r_1$ and $r_2$ that satisfy $\epsilon_1=\epsilon_2$ produce a graph that is identical to the starting graph. If, on the other hand, D moves on $r_3$ and leaves it as-is, then C should perform a smoothing swap on $r_4$. The resulting graph is isomorphic to the original graph. Any other pair of choices following the $\epsilon_3\neq\epsilon_4$ strategy similarly produces an isomorphic graph.
\end{proof}

\begin{proposition}  Suppose the Region Smoothing Swap Game is played on a connected starting board determined by the Borromean rings (pictured in Figure \ref{Borromean}). Then Player C has a winning strategy when playing second.
\end{proposition}

\begin{figure}[htbp] \begin{center}
 \begin{tabular}{ccc}
\includegraphics[scale=.5]{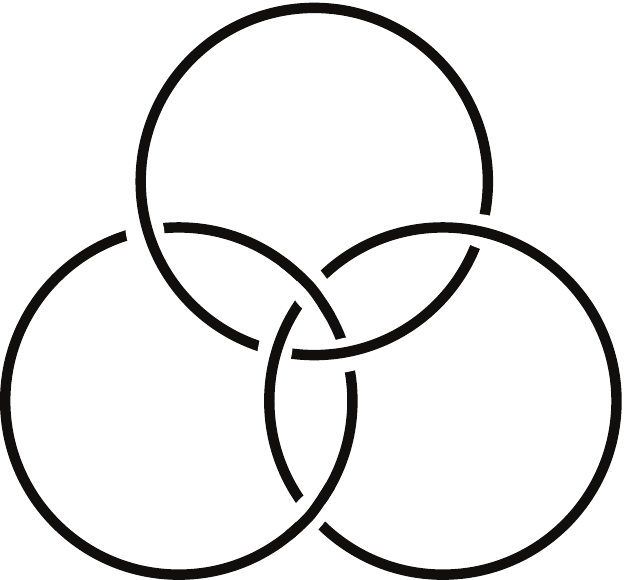}& &  \includegraphics[scale=.5]{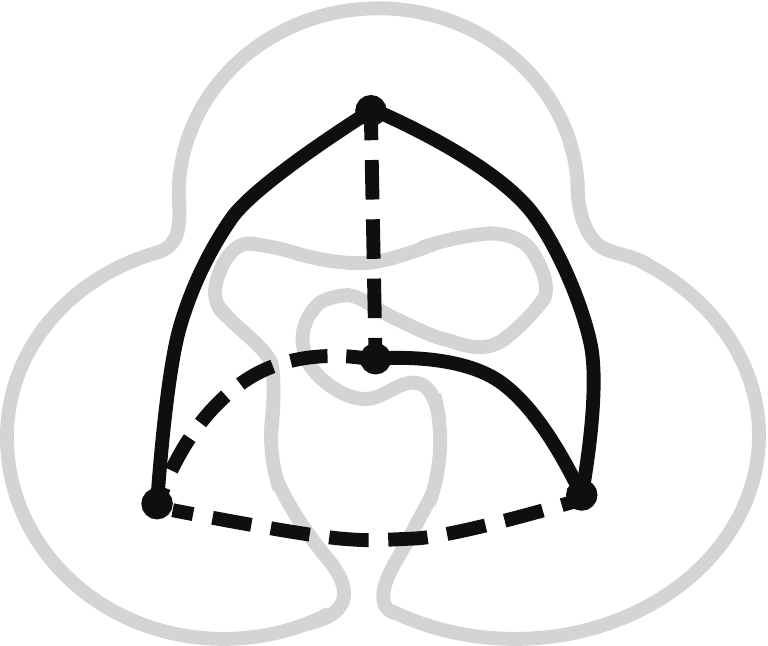} 
\end{tabular}
\caption{{\bf The Borromean rings and a corresponding game board.}}
\label{Borromean}\end{center}
\end{figure}

\begin{figure}[htbp] \begin{center}
 \begin{tabular}{ccc}
(a) \includegraphics[scale=.5]{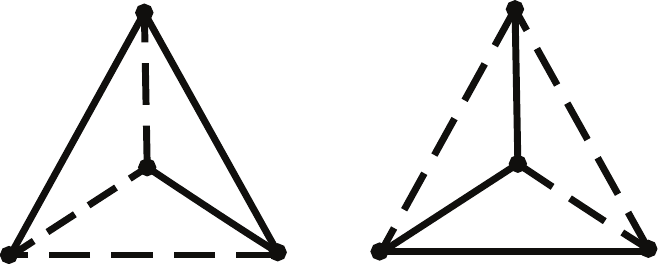}& & (b) \includegraphics[scale=.5]{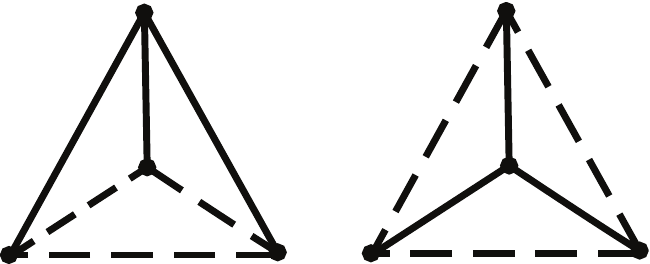} 
\end{tabular}
\caption{{\bf All connected starting game boards for the Borromean rings, up to rotation and reflection.}}
\label{BorroGameStart}\end{center}
\end{figure}

\begin{proof}  We begin with the standard representation of the Borromean rings and the connected game board shown in  Figure~\ref{Borromean}.  Notice that the associated graph---which we see is the complete graph on four vertices, $K_4$---contains two types of spanning trees: those that contain a vertex of degree three (shown in Figure~\ref{BorroGameStart}(b)) and those that do not (shown in Figure~\ref{BorroGameStart}(a)).

\begin{figure}[htbp] \begin{center}
 \begin{tabular}{ccccc}
\includegraphics[scale=.4]{Game2GraphBorromeanRings-eps-converted-to.pdf}& &  \includegraphics[scale=.4]{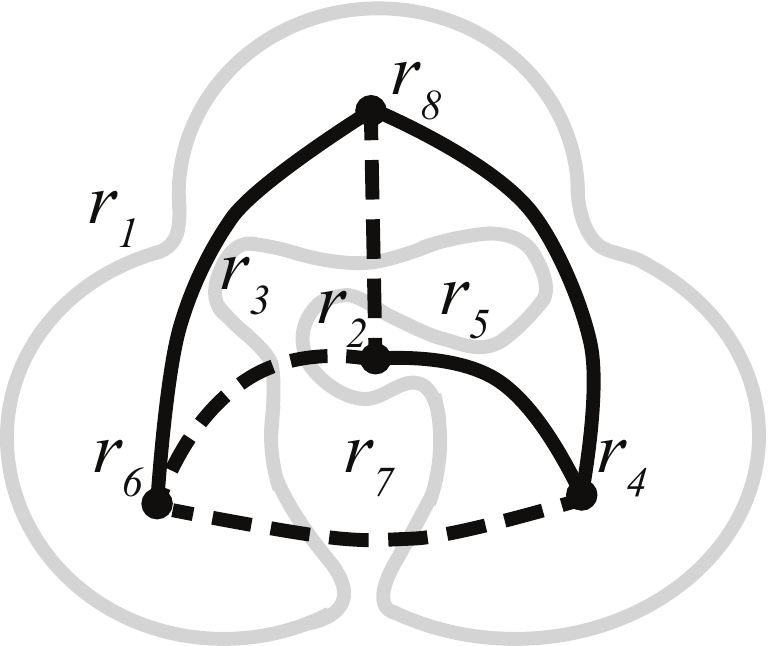} & &  \includegraphics[scale=.4]{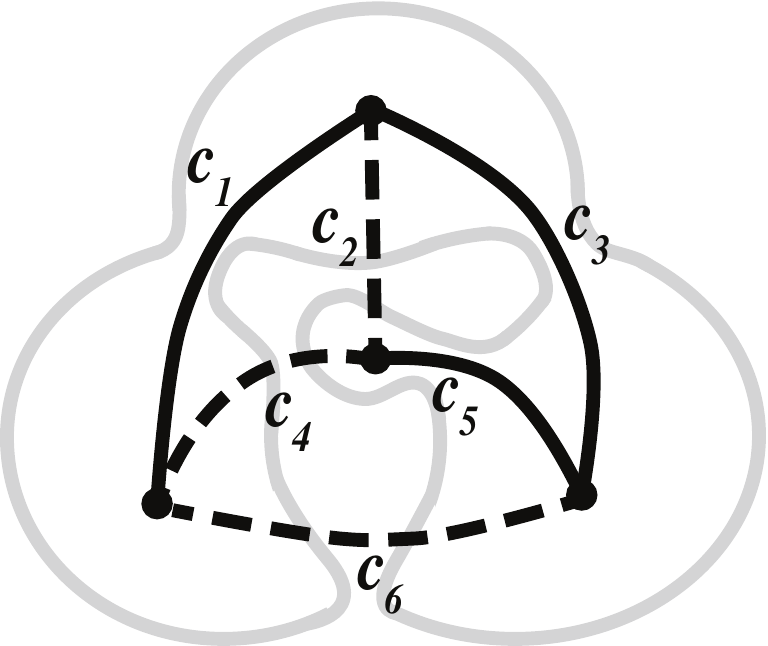} 
\end{tabular}
\caption{{\bf The Borromean rings and a corresponding game board with region/vertex labels and edge labels.}}
\label{BorroNotation}\end{center}
\end{figure}

For the graphs in Figure~\ref{BorroGameStart}(a), a pairing strategy exists that results in a win for Player C playing second.  We describe this strategy using the notation defined in Figure~\ref{BorroNotation}. The defensive strategy for C is to respond to D's moves so that $\epsilon_1 =\epsilon_2$, $\epsilon_3=\epsilon_4$, $\epsilon_5=\epsilon_6$ and $\epsilon_7=\epsilon_8$.  

For each pair of moves, this mimicking strategy results in either changing every edge of $K_4$, when $\epsilon_{2i} =\epsilon_{2i-1} = 1,$ or changing none of the edges in the graph, when $\epsilon_{2i} =\epsilon_{2i-1} = 0$.  For either of the two graphs in Figure~\ref{BorroGameStart}(a), changing every edge of the graph produces the complementary subgraph of the checkerboard graph, which is the other graph from Figure~\ref{BorroGameStart}(a); similarly for rotations and reflections of these graphs.  
 
For the graphs in (b), with a vertex of degree three, the mimicking strategy 
employed above does not work.  This strategy fails because the spanning trees in Figure~\ref{BorroGameStart}(b) have complements that are not trees. Thus, the corresponding game board does not remain connected when moves such as $\epsilon_1 =\epsilon_2=1$ are made.   However, there is a shuffling of the pairings that does result in a win for C, namely $\epsilon_1=\epsilon_7, \epsilon_2 = \epsilon_8, \epsilon_3=\epsilon_5,$ and $\epsilon_4=\epsilon_6$.  To see why this pairing works on either starting board from Figure~\ref{BorroGameStart}(b), we first make the observation shown in Figure~\ref{BorroStartb-reply}.  This figure shows the ``star trek" graph---our name for  the graph in Figure~\ref{BorroStartb-reply}(c)---that when added to either spanning tree that contains a vertex of degree three, yields the other spanning tree that contains a vertex of degree three.  Next, we note that the pairing $\epsilon_1=\epsilon_7=1$ results in adding the star trek graph to the connected starting graph.  The pairings $ \epsilon_2 = \epsilon_8, \epsilon_3=\epsilon_5,$ and $\epsilon_4=\epsilon_6$ also result in adding the star trek graph.  Therefore, this strategy always yields one of the connected spanning trees from Figure~\ref{BorroStartb-reply}(b).
 \end{proof}

\begin{figure}[htbp] \begin{center}
 \begin{tabular}{ccc}
(b) \includegraphics[scale=.5]{Game2GraphBorromeanStart3-eps-converted-to.pdf}& & (c) \includegraphics[scale=.5]{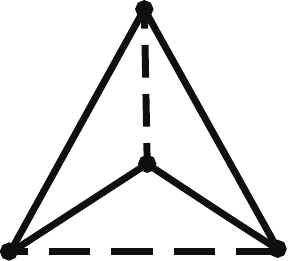} 
\end{tabular}
\caption{{\bf For either graph in (b), when we add the graph in (c) the result is the other graph in (b).  We call the graph in (c) a ``star trek" graph due to it's similarity to the Starfleet insignia from the Star Trek$^{\copyright}$ television and movie series.}}
\label{BorroStartb-reply}\end{center}
\end{figure}

\section{Game Play on the Klein Bottle}

So far, we've focused on game boards that live on the surface of a sphere, but some of the most interesting examples of games we've come across live on more complex surfaces. In particular, there exists an infinite family of game boards on the Klein bottle on which Player C has a winning strategy moving second. Before we describe our game board, let us recall one useful method of representing the Klein bottle: as a square with certain edges identified. In Figure \ref{Klein}, we see a popular representation of a Klein bottle (from Wikipedia) together with two possible polygonal representations, where edges are identified in pairs according to the orientations shown in the diagram. 

\begin{figure}[htbp] \begin{center}
(a) \includegraphics[height=1.5in]{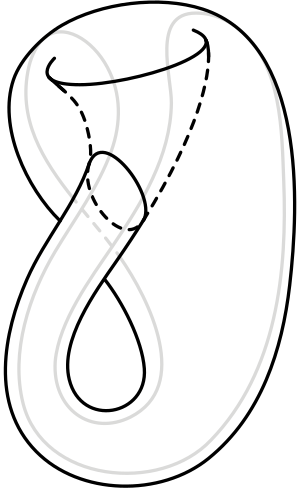}\hspace{.7in}(b) \includegraphics[height=1.2in]{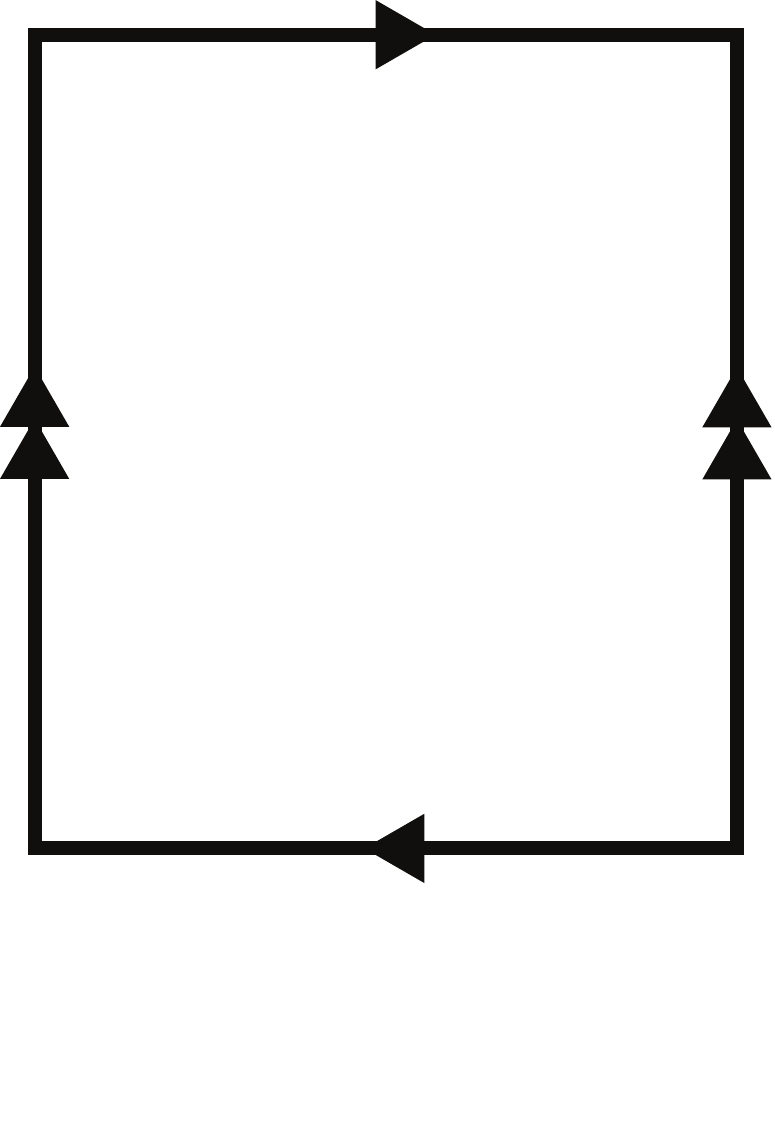}\hspace{.4in} (c) \includegraphics[height=1.2in]{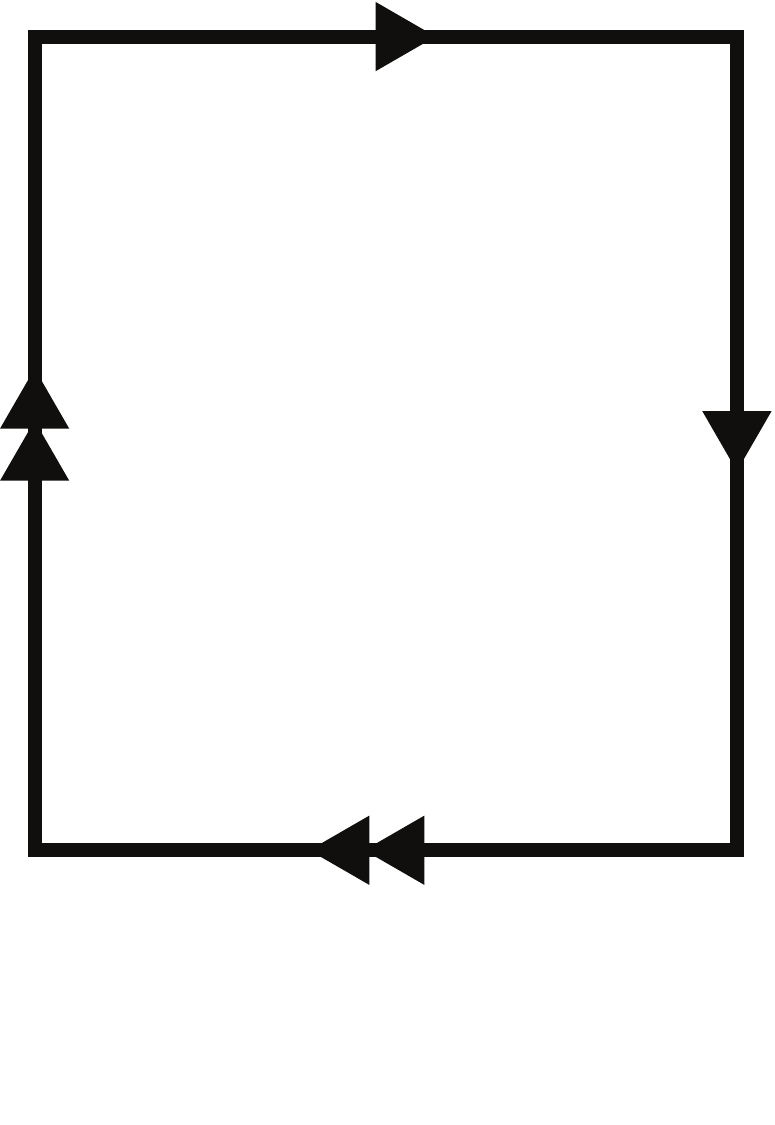}
\caption{{\bf A Klein bottle (a), and two polygonal representations of a Klein bottle (b) and (c).}}
\label{Klein}\end{center}
\end{figure}

To picture a family of game boards on the Klein bottle that we call the ``ladder family," we make use of the polygonal representation in Figure~\ref{Klein}(b).   The two smallest members of this family, $n=2$ (the 2-step ladder) and $n=3$ (the 3-step ladder), are shown in Figure \ref{KleinGame}. We will refer to the horizontal graph edges as the ``steps" of the ladder and the vertical edges as the  ``rails." 

\begin{figure}[htbp] \begin{center}
 \begin{tabular}{ccccc}
region/vertex labels&&edge labels&&smoothed state\\
\hline\\
\includegraphics[width=1.3in]{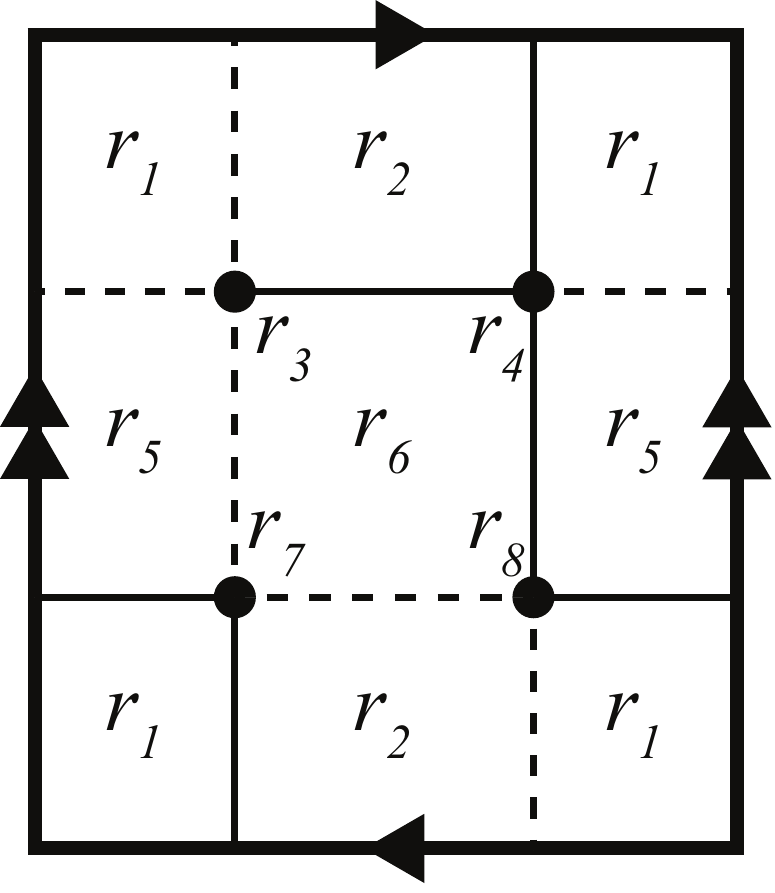}&&\includegraphics[width=1.3in]{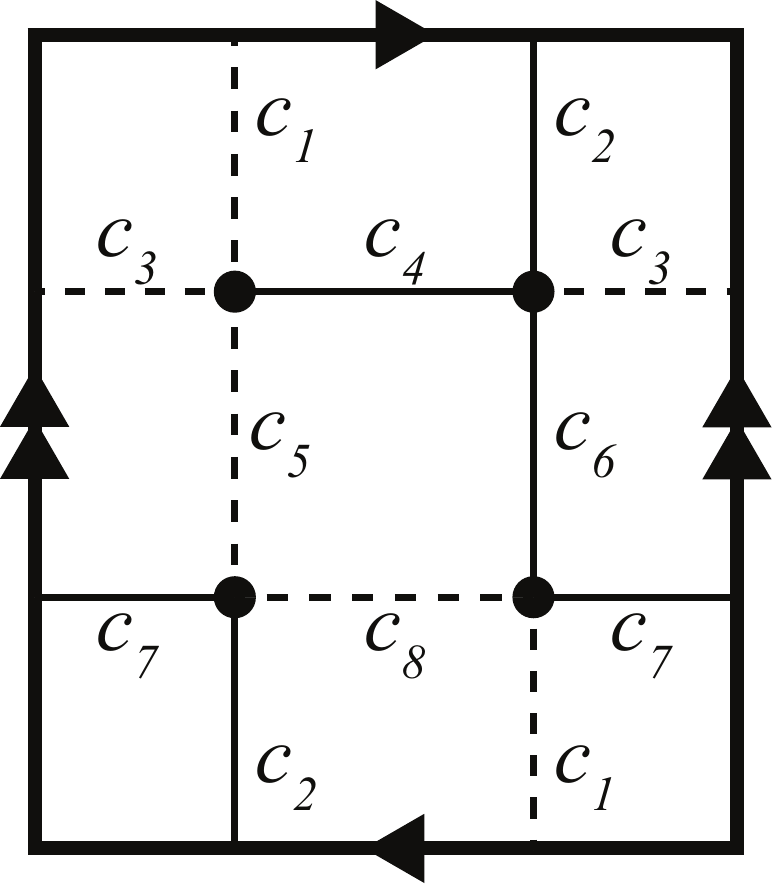}&&\includegraphics[width=1.3in]{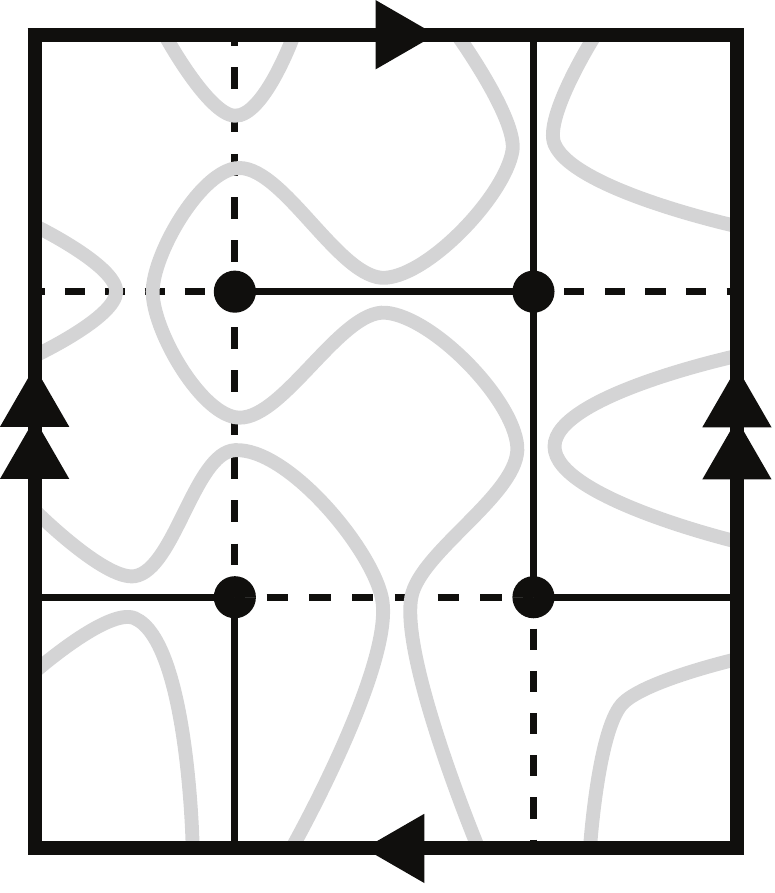}\\

\hline\\
\includegraphics[width=1.3in]{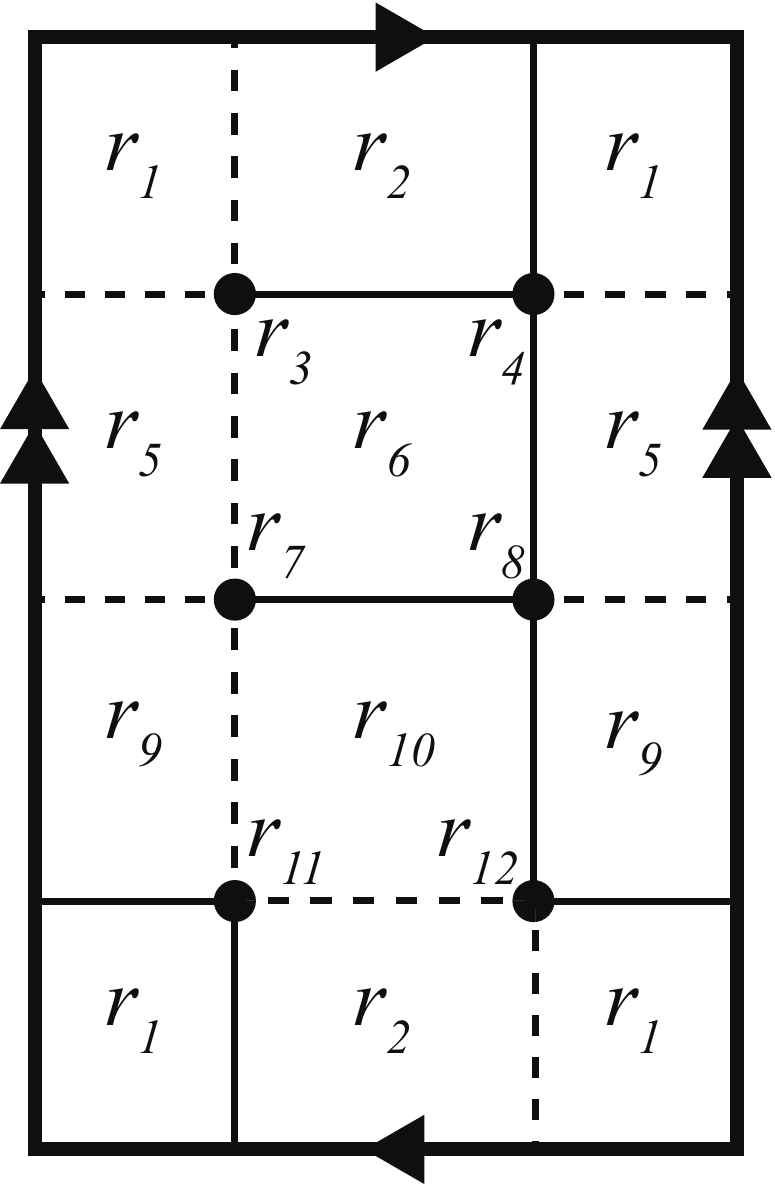}&&\includegraphics[width=1.3in]{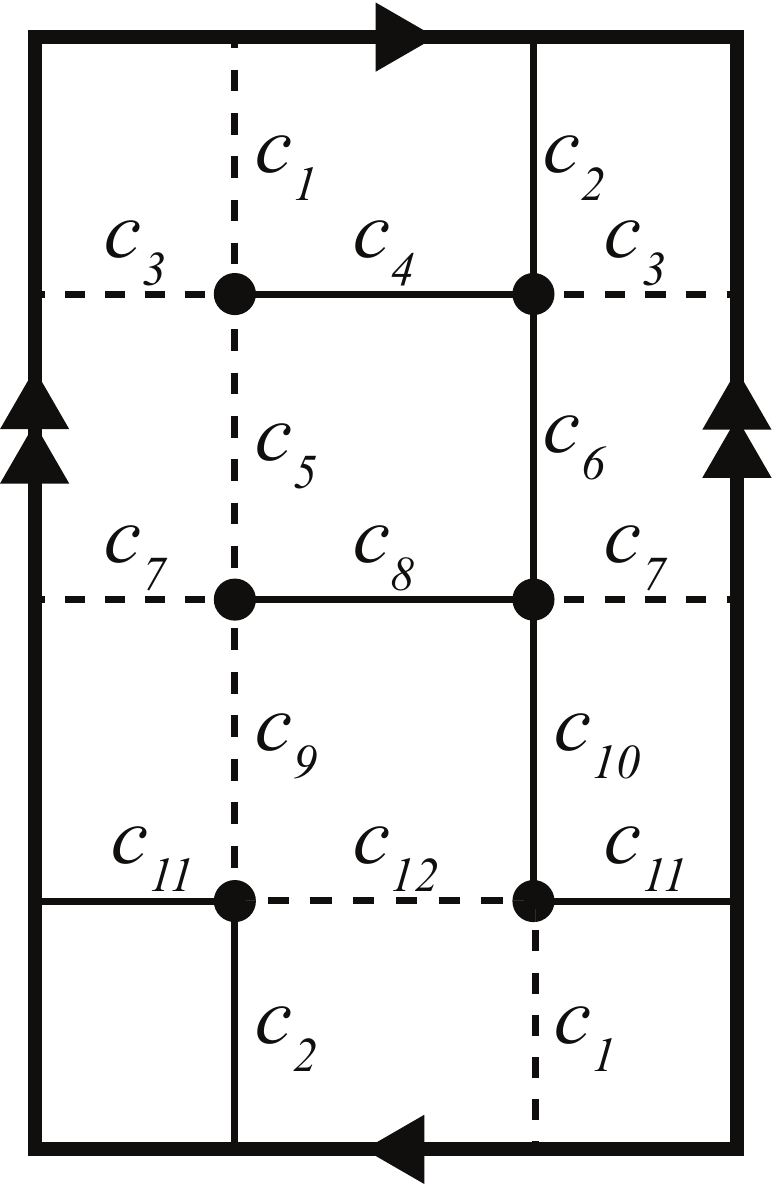}&&\includegraphics[width=1.3in]{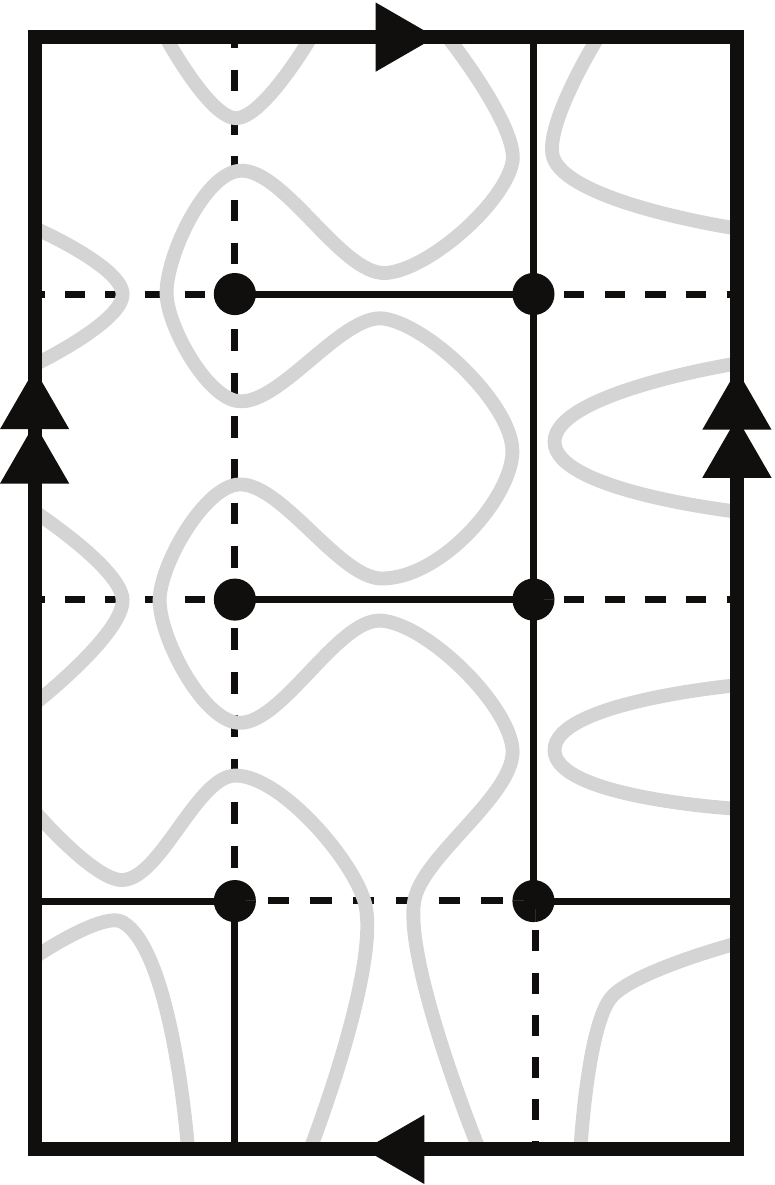}
\end{tabular}
\caption{{\bf Two members---the 2-step ladder (above) and the 3-step ladder (below)---of a family of game boards on the Klein bottle, with region/vertex and crossing labels and corresponding smoothed state.}}
\label{KleinGame}\end{center}
\end{figure}

Unlike game play on the sphere, notice that it is possible for a smoothed state on the Klein bottle to be connected when the selected subgraph contains a {\it cycle}.    For example, the smoothed states in Figure~\ref{KleinGame} are connected, but the graphs contain a cycle. Before further discussion of the game, we make an observation about special game boards, like those in Figure~\ref{KleinGame}, that contain a cycle and always result in a connected smoothed state.

\begin{lemma}\label{ladderlemma} If a subgraph of an $n$-step ladder graph on the Klein bottle contains exactly one horizontal step and exactly one rail at each height of the ladder,  then the smoothed state corresponding to this subgraph is connected. 
\end{lemma}

\begin{proof}  We begin by making a quick observation regarding a simple subgraph of the $n$-step ladder graph. Then we prove the lemma for an arbitrary subgraph of the ladder that contains exactly one horizontal step and exactly one rail at each height.

The simple subgraph we consider first is the graph that contains the edges $c_2, c_6, \dots, c_{4n-2}$ and $c_{4n}$.  For the 3-step ladder, this graph is in Figure~\ref{LemmaGraphs}(a).   Observe that an $\epsilon$-neighborhood of this graph within the Klein bottle is a M\"obius strip, therefore the boundary of this $\epsilon$-neighborhood consists of exactly one connected component.  Also notice that if smoothings were made around this graph they would trace out the boundary of the M\"obius strip, thus one smoothed component would surround this graph.

\begin{figure}[htbp] \begin{center}
 \begin{tabular}{ccccc}

\includegraphics[width=1.1in]{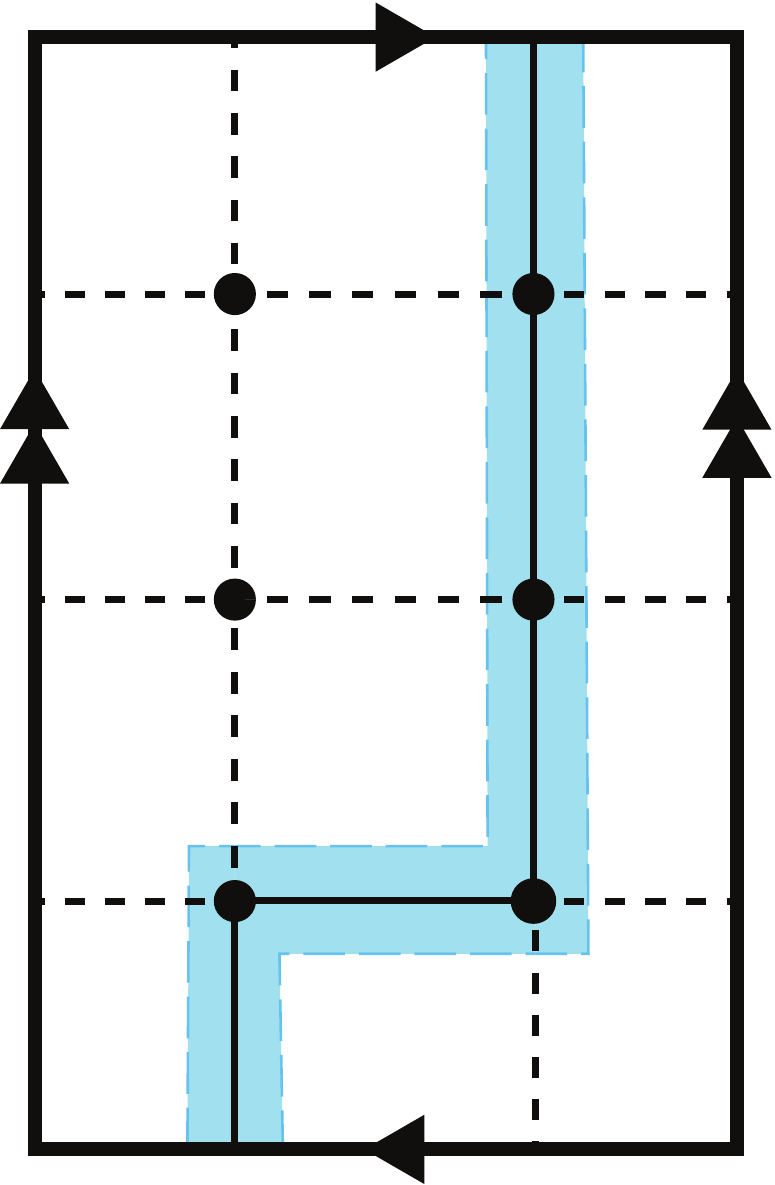}&&\includegraphics[width=1.1in]{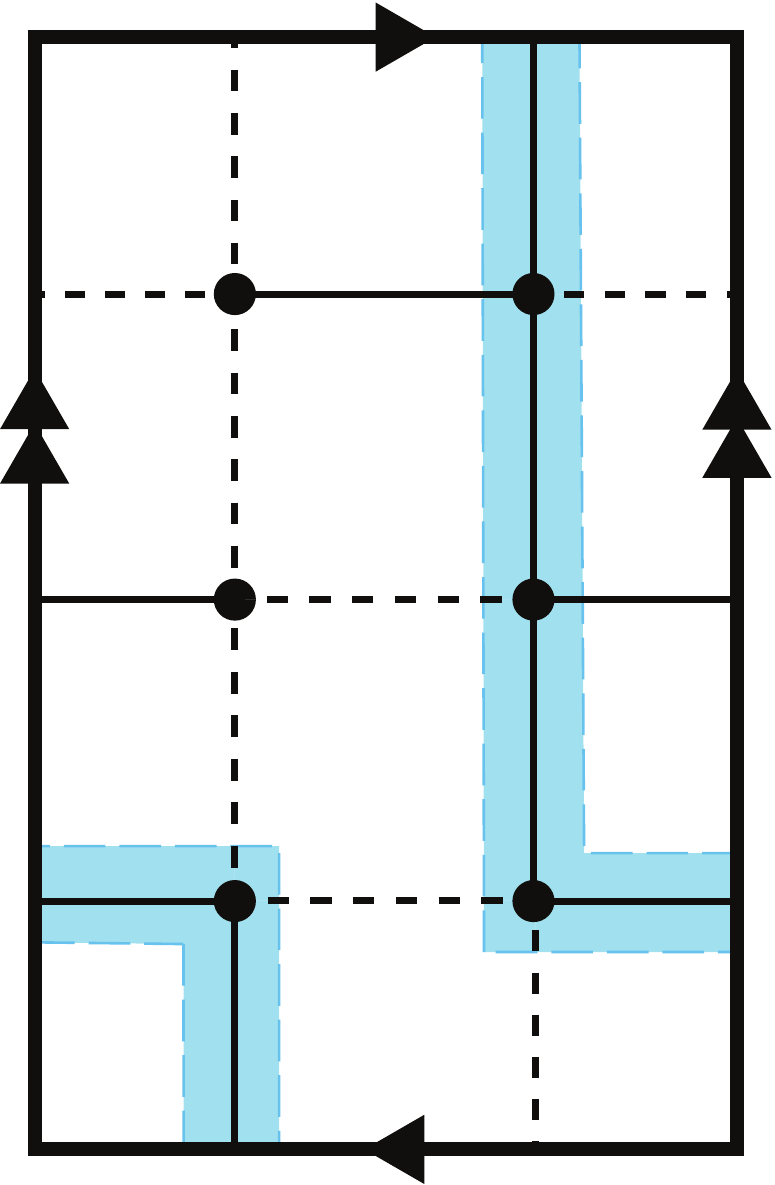}&&\includegraphics[width=1.1in]{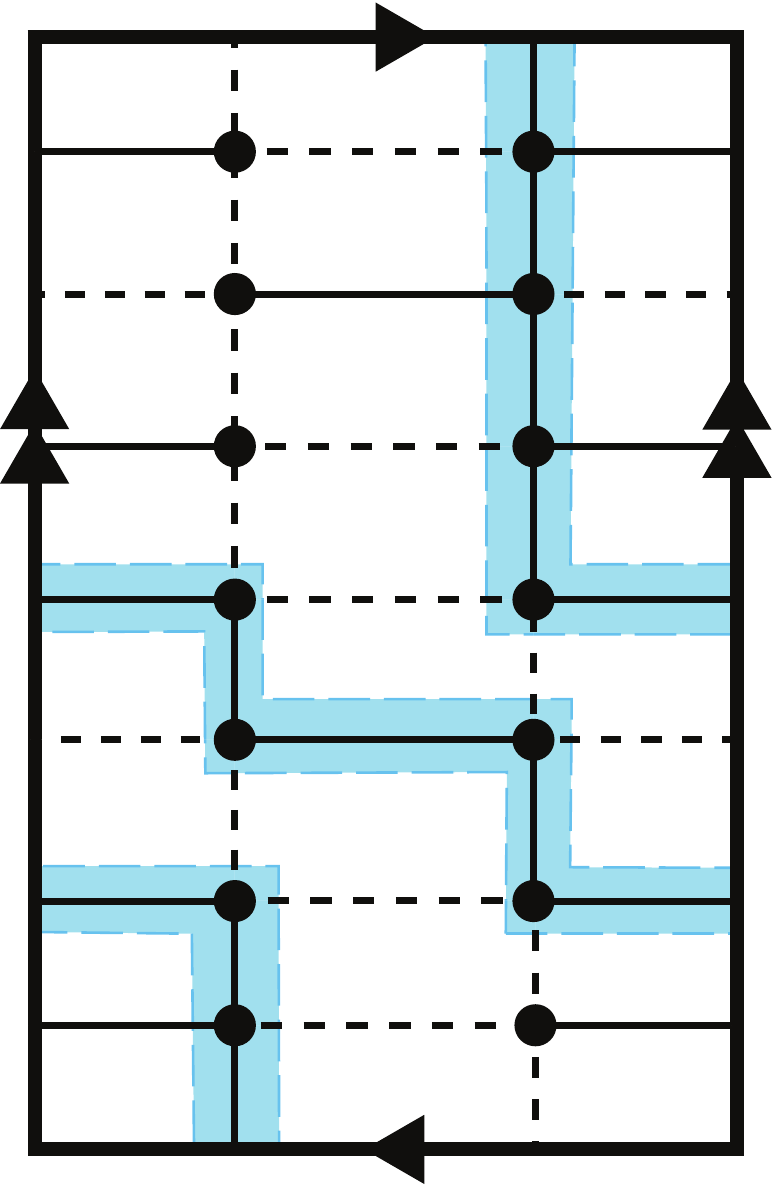} \end{tabular}
\caption{{\bf (a) A simple graph with $\epsilon$-neighborhood homeomorphic to a M\"obius strip. (b-c) Examples of subgraphs of the ladder graph that contain exactly one horizontal step and exactly one rail at each height of the ladder.}}
\label{LemmaGraphs}\end{center}
\end{figure}

To prove Lemma~\ref{ladderlemma}, we consider an arbitrary subgraph, $G$, of the ladder that contains exactly one horizontal step and exactly one rail at each height.  Since the graph contains exactly one step at each height, it cannot contain any vertices of degree four or zero.  Thus, $G$ can only contain vertices of degree one, two, or three, as shown in Figure~\ref{LemmaGraphs}(b) \& (c).    In particular, this implies that every vertex is included in $G$.

Since $G$ contains a single rail at each height, each step of $G$ will either have a rail on each end or two rails on one end of the step.  Thus, each step in $G$ will either have both adjacent vertices of degree two or the step will have vertices of degree three and one.  We create an $\epsilon$-neighborhood around the subgraph $G'$ of $G$ that contains every rail of $G$ and each step of $G$ whose adjacent vertices are both degree two.  Using an inductive argument starting at the bottom of the ladder, this subgraph can be viewed as a non-decreasing path up the ladder that connects up exactly once along the orientation-reversing identifications on the bottom and top of the polygon.  If the $\epsilon$-neighborhood around the subgraph $G'$ is glued along all orientation-preserving identifications within $G'$, we get a connected rectangular strip with orientation-reversing gluings at each end.  Therefore, the $\epsilon$-neighborhood around the subgraph $G'$ is homeomorphic to a  M\"obius strip.  As in the case of the simple graph, smoothings made in the diagram according to the edges in $G'$ will result in a single boundary component around $G'$.  
  
To complete the argument for $G$, we notice that the only edges in $G$ that are not in $G'$ are the degree three and one vertices.  Thus, the graph $G'$ is a connected strong deformation retract of $G$ via the homotopy that continuously shrinks each step in $G - G'$ towards the degree three vertex of that step.  This simple homotopy can be extended to the $\epsilon$-neighborhood of $G$ to prove that the $\epsilon$-neighborhood of $G'$ is a strong deformation retract of the $\epsilon$-neighborhood of $G$. Following this homotopy on $\epsilon$-neighborhoods, the smoothed state corresponding to the graph $G$ is isotopic to the smoothed state corresponding to $G'$.  Therefore, the smoothed state corresponding to $G$ is one connected component.
    \end{proof}

\begin{theorem} If Player C moves second, then Player C has a winning strategy playing the Region Smoothing Swap Game on any game board associated to ladder family of links on the Klein bottle that contains exactly one horizontal step and exactly one rail at each height of the ladder (such as the game boards shown in Figure~\ref{KleinGame}).  
\end{theorem}

\begin{proof} First, following Lemma~\ref{ladderlemma}, we note that our goal is to produce a subgraph of the checkerboard graph that contains exactly one rail and exactly one step at each height. Our starting configuration has this property. For the game on this graph, we follow a pairing strategy, setting $\epsilon_{2i-1}=\epsilon_{2i}$ (see labelings in Figure \ref{KleinGame}) so Player C will mimic D's moves in these pairs. 

First of all, if Player D chooses $\epsilon_{2i-1}=0$ or $\epsilon_{2i}=0$ for some $i$, our mimicking strategy preserves the status quo---our graph is unaltered. Suppose, then, that Player D chooses $\epsilon_{2i-1}=1$ or $\epsilon_{2i}=1$ for some $i$. Then Player C will ensure $\epsilon_{2i-1}=\epsilon_{2i}=1$. In the diagram that results from this pair of moves, neither edge $c_{2i-1}$ or $c_{2i}$ has been altered, but four edges have been switched, namely $c_{2i-3}$, $c_{2i-2}$, $c_{2i+1}$, and $c_{2i+2}$ (with subscripts mod $4n$ for the $n$-step ladder). If exactly one of the edges $c_{2i-3}$ and $c_{2i-2}$ was initially in the (bold) subgraph, then exactly one edge will be in the subgraph following the two moves. Similarly for $c_{2i+1}$ and $c_{2i+2}$. Thus, our desired ``one edge at each height" property has been preserved, and the final game board is connected, by Lemma~\ref{ladderlemma}.
\end{proof}

\section{Game Play on the Connect Sum of Two Klein Bottles}

We've just seen an interesting game board on a Klein bottle, but why stop there? Are there game boards on more complex surfaces for which Player C has a winning strategy? We asked just this question and found an intriguing example on the connect sum of two Klein bottles. 

\begin{figure}[htbp] \begin{center}
\includegraphics[height=2in]{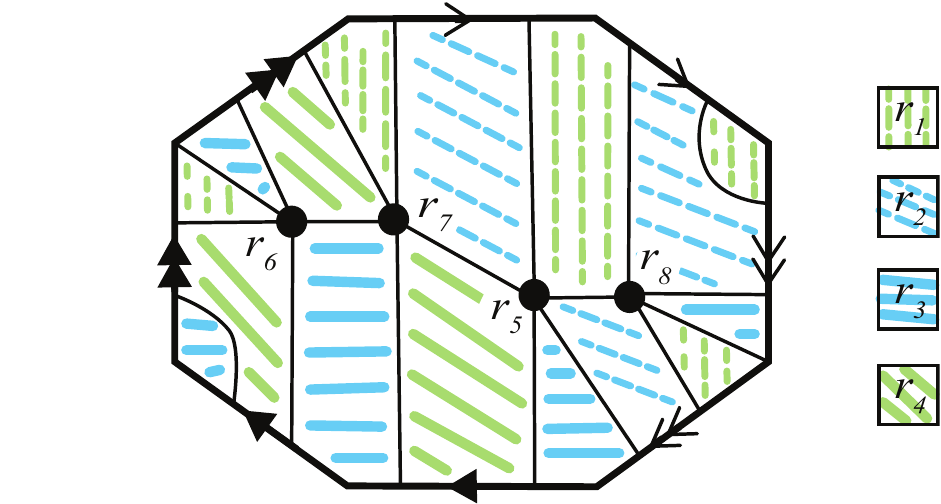}\\
\includegraphics[height=2in]{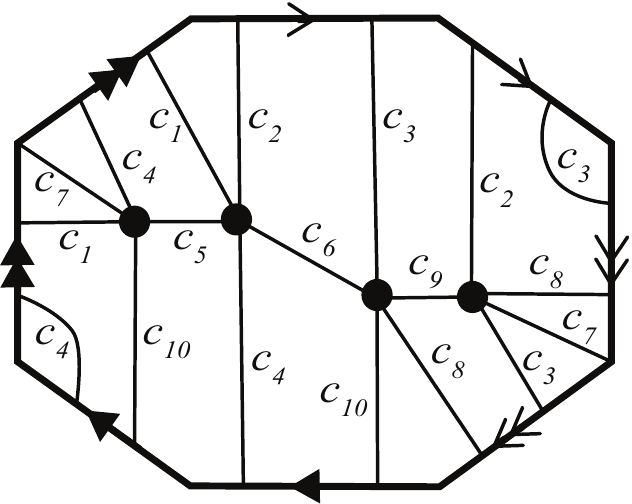}
\caption{{\bf A graph on the connect sum of two Klein bottles, with labels.}}
\label{TwoKlein}\end{center}
\end{figure} The matrix corresponding to this example is as follows, where the $i$th column corresponds to $c_i$ and the $j$th row corresponds to $r_j$.

\begin{gather*}
\begin{bmatrix} 
1&1&1&0&0&0&1&0&1&0\\ 
0&1&1&0&0&1&0&1&1&0\\ 
0&0&0&1&1&0&1&1&0&1\\ 
1&0&0&1&1&1&0&0&0&1\\ 
0&0&1&0&0&1&0&1&1&1\\ 
1&0&0&1&1&0&1&0&0&1\\ 
1&1&0&1&1&1&0&0&0&0\\ 
0&1&1&0&0&0&1&1&1&0\\ 
\end{bmatrix}
\end{gather*}

\begin{theorem} If Player C moves second, then Player C has a winning strategy playing the Region Smoothing Swap Game on the graph on the connect sum of two Klein bottles shown in Figure~\ref{TwoKlein} when the starting game graph corresponds to any of the $V_0$ vectors below.  

\begin{gather*}
V_0^1=\begin{bmatrix} 0 & 1 & 0 & 0 & 1 & 1 & 0 & 0 & 1 & 1\end{bmatrix}\\
V_0^2=\begin{bmatrix} 1 & 0 & 1 & 1 & 0 & 0 & 1 & 1 & 0 & 0\end{bmatrix}\\
V_0^3=\begin{bmatrix} 0 & 1 & 0 & 0 & 1 & 0 & 1 & 0 & 1 & 1\end{bmatrix}\\
V_0^4=\begin{bmatrix} 1 & 0 & 1 & 1 & 0 & 1 & 0 & 1 & 0 & 0\end{bmatrix}\\
\end{gather*}
\end{theorem}

\begin{proof} First, note that each of the four $V_0^k$ vectors above corresponds to a connected game board. The reader may verify this fact by checking that an $\epsilon$-neighborhood of each graph associated to a $V_0^k$ has a single boundary component.  With this in mind, the strategy we describe for Player C is the following mimicking strategy: $\epsilon_1 =\epsilon_3$, $\epsilon_2 =\epsilon_6$, $\epsilon_4 =\epsilon_8$, and $\epsilon_5 =\epsilon_7$.

Note that, if a region smoothing swap is performed at both $r_2$ and $r_6$, the effect is that {\em every} $c_i$ is swapped. The same is true for the pair $r_4$ and $r_8$. On the other hand, if a region smoothing swap is performed at both $r_1$ and $r_3$, all but $c_6$ and $c_7$ will be swapped. The same is true for $r_5$ and $r_7$. 

Now, if we begin with $V_0^1$ and all $c_i$'s are swapped, the resulting vector is $V_0^2$ (and vice versa). Similarly for $V_0^3$ and $V_0^4$. Moreover, if all but $c_6$ and $c_7$ are swapped in $V_0^1$, the result is $V_0^3$. Similarly for $V_0^2$ and $V_0^4$. Together, this implies that if Player C follows a mimicking strategy in the pairs of $r_i$'s designated above starting with any of the $V_0^k$ vectors, any pair of moves produces a game board that is one of the four $V_0^k$'s. Since these four vectors all represent connected game boards, Player C has a winning strategy.
\end{proof}

\section{Related Results}\label{sec:related}
As we discussed in the introduction, the Region Smoothing Swap Game is a variation on the Link Smoothing Game. In \cite{link-smooth}, game boards were partially classified according to their outcome classes. Here, we provide a proof of completing the classification of Link Smoothing game boards.

\begin{theorem}\label{link-smoothing-final} Let $G$ be a connected, planar graph associated to a link shadow $D$.  If $G$ represents a $\mathcal{P}$-position game (i.e., a game in which the second player has a winning strategy), then $G$ is composed of two edge-disjoint spanning trees.
\end{theorem}

To prove the theorem, we use a result of Nash-Williams~\cite{NW} and separately Tutte~\cite{Tutte},  that gives a necessary and sufficient condition for a graph $G$ to have two edge-disjoint spanning trees.  Before stating the conditions, though, we need some notation for a special graph associated to $G$ that is defined in terms of a partition of the vertex set of $G$.  

Let $G$ be a graph with vertex set $V(G)$ and edge set $E(G)$.  We denote the number of vertices and edges in $G$ by $|V(G)|$ and  $|E(G)|$ respectively. For a partition $P$ of $V(G)$, we define $E_P(G)$ as the set of edges that join vertices belonging to different members of $P$.  The graph $G_P$ is defined as the graph with vertex set $P$ and edge set $E_P(G)$.  The Nash-Williams and Tutte result can now be stated as follows.

\begin{theorem}[Nash-Williams, Tutte] A graph $G$ has $k$ edge-disjoint spanning trees if and only if 
$$|E_P(G)| \geq k (|P|-1)$$
for every partition $P$ of  $V(G).$
\end{theorem}

\begin{proof}[Proof of Theorem~\ref{link-smoothing-final}]  We begin with the supposition that $G$ is a $P$-position graph for the link smoothing game. The definition of  $P$-position implies that the player with the goal to keep the diagram connected has a winning strategy when moving second on the given game board.  Such a winning strategy can only exist if there are an even number of edges in the graph $G$, else the player moving last would be the player with goal to disconnect and such a player can always disconnect the diagram on the final move.

By way of contradiction, we suppose that $G$ does not consist of two edge disjoint spanning trees.  Then the result of Nash-Williams and Tutte implies the existence of a partition $P$ of the vertices of $G$ such that $|E_P(G)| < 2(|P| -1)$.  Since the set $E_P(G)$ is a subset of the edges of $G$ and $|P|$ is less than $|V(G)|$, we can conclude  $|E(G)|\leq |E_P(G)| < 2(|P| -1) \leq 2(|V(G)| -1).$  Hence,  $$|E(G)| <  2(|V(G)| -1).$$  By Theorem 6 in~\cite{link-smooth}, the previous inequality implies the graph represents an $L$-position diagram.  This contradicts our $P$-position supposition.
\end{proof}

\section{Questions for Further Research}

In doing research on the Region Smoothing Swap Game, our goal has been to find examples of link diagrams on surfaces on which Player C has a winning strategy. Since Player D tends to have an advantage in this game, finding such examples can be challenging. What is special about examples for which Player C has a winning strategy? An analysis akin to the one begun for the Link Smoothing Game in \cite{link-smooth} which was completed in Section \ref{sec:related} above would be interesting. Such an analysis will likely be more challenging to perform, however, in the setting of the Region Smoothing Swap Game on surfaces. It would be especially interesting to know more about the relationship between examples on which C can win and the surfaces on which these link smoothings live.

This topic provides a wealth of other open problems for those who are curious about variations on the Region Smoothing Swap Game. Just as with the Link Smoothing Game, the players' goals can be changed or the allowable moves can be modified to create a new game to study. We encourage our readers to invent and study their own games with knots, links, graphs, and surfaces!

 \section{Acknowledgements.}  The authors would like to thank the Simons Foundation (\#426566, Allison Henrich) for their support of this research.

\bibliographystyle{amsalpha}

\begin{thebibliography}{99}

\bibitem{RUG}{S. Brown, F. Cabrera, R. Evans, G. Gibbs, A. Henrich and J. Kreinbihl, The region unknotting game. {\it Math. Mag.} {\bf 90} no. 5 (2017) 323-337.}

\bibitem{link-smooth}{A. Henrich and I. Johnson, The link smoothing game. {\it AKCE Int. J. Graphs Comb.} {\bf 9} no. 2 (2012) 145-159.}

\bibitem{NW}{C. St.J. A. Nash-Williams, Edge-Disjoint Spanning Trees of Finite Graphs. {\it Journal London Math. Soc.} {\bf 36} (1961) 445--450.}

\bibitem{Tutte}{W. T. Tutte, On the problem of decomposing a graph into $n$ connected factors. {\it Journal London Math. Soc.} {\bf 142} (1961), 221--230. } 

\end{thebibliography}

 \end{document}